\documentclass[11pt,a4paper, francais]{smfart}
\usepackage{ifthen}
\usepackage{amsmath}
\usepackage{amssymb}
\usepackage{latexsym}
\usepackage{bbm}
\usepackage{graphicx}
\usepackage{pdfsync}
\usepackage{verbatim}
\usepackage{fancyhdr}
\usepackage{url}  
\usepackage[T1]{fontenc}             
\usepackage[french]{babel}
\usepackage[applemac]{inputenc}

\DeclareMathOperator*{\tr}{\text{Tr}}

\newcommand{\be}{\begin{equation}}
\newcommand{\ee}{\end{equation}}
\newcommand{\ba}{\begin{aligned}}
\newcommand{\ea}{\end{aligned}}

\newtheorem{lem}{Lemme}[section]
\newtheorem{thm}{Th\'eor\`eme}
\newtheorem{prop}[lem]{Proposition}
\newtheorem{cor}[lem]{Corollaire}

\newtheorem{remk}[lem]{Remarque}
\newtheorem*{thm*}{Th\'eor\`eme}
\numberwithin{equation}{section}
\numberwithin{figure}{section}

\begin{document}


\title[Sobolev probabiliste]{Injections de Sobolev probabilistes et applications}
\author{Nicolas Burq}
\address{Laboratoire de Math\'ematique d'Orsay, Universit\'e Paris-Sud, UMR 8628 du CNRS, 91405 Orsay Cedex, France} \author{Gilles Lebeau}
\address{D\'epartement de Math\'ematiques, Universit\'e de Nice Sophia-Antipolis
 Parc Valrose 06108 Nice Cedex 02,   France et Institut Universitaire de France
}
\email{lebeau@unice.fr}

\begin{abstract}
On d\'emontre dans cet article des versions probabilistes des injections de Sobolev sur une vari\'et\'e riemannienne compacte, $(M,g)$. Plus pr\'ecisement on d\'emontre que pour des mesures de probabilit\'e naturelles sur l'espace $L^2(M)$, presque toute fonction appartient \`a tous les espaces $L^p(M)$, $p<+\infty$. On donne ensuite des applications \`a l'\'etude des harmoniques sph\'eriques sur la sph\`ere $\mathbb{S}^d$: on d\'emontre (encore pour des mesures de probabilit\'e naturelles) que presque toute base Hilbertienne de $L^2( \mathbb{S}^d)$ form\'ee d'harmoniques sph\'eriques a tous ses \'el\'ements uniform\'ement born\'es dans tous les espaces $L^p(\mathbb{S}^d), p<+\infty$ On d\'emontre aussi des r\'esultats similaires sur les tores $\mathbb{T}^d$. On donne aussi une application \`a l'\'etude du taux de d\'ecroissance de l'\'equation des ondes amortie dans un cadre o\`u la condition de contr\^ole g\'eom\'etrique de Bardos, Lebeau et Rauch n'est pas v\'erifi\'ee. En supposant le flot ergodique, on d\'emontre qu'il existe sur des ensembles de mesure arbitrairement proche de $1$ (dans l'espace des donn\'ees initiales d'\'energie finie), un taux de d\'ecroissance uniforme. Finalement, on conclut avec  une application \`a l'\'etude de l'\'equation des ondes semi-lin\'eaire $H^1$-surcritique, pour laquelle on d\'emontre que pour presque toute donn\'ee initiale, les solutions faibles sont fortes et uniques (localement en temps).
\end{abstract}
\begin{altabstract}In this article, we give probabilistic versions of Sobolev embeddings on any Riemannian manifold $(M,g)$. More precisely, we prove that for natural probability measures on  $L^2(M)$, almost every function belong to all spaces  $L^p(M)$, $p<+\infty$. We then give applications to the study of the growth of the $L^p$ norms of spherical harmonics on spheres $\mathbb{S}^d$: we prove (again for natural probability measures) that almost every Hilbert base of $L^2( \mathbb{S}^d)$ made of spherical harmonics has all its elements uniformly bounded in all $L^p(\mathbb{S}^d), p<+\infty$ spaces. We also prove similar results on tori $\mathbb{T}^d$. We give then an application to the study of the decay rate of damped wave equations in a frame-work where the geometric control property on Bardos-Lebeau-Rauch is not satisfied. Assuming that it is violated for a measure $0$ set of trajectories, we prove that there exists almost surely a rate. Finally, we conclude with an application to the study of the  $H^1$-supercritical wave equation, for which we prove that for almost all initial data, the weak solutions are strong and unique, locally in time.
\end{altabstract}
\maketitle
\tableofcontents
\section{Introduction}\label{sec0}
L'objet de cet article est de d\'emontrer que si on choisit des fonctions au hasard sur une vari\'et\'e compacte, pour des mesures de probabilit\'e  naturelles, alors il est possible d'am\'eliorer grandement les injections de Sobolev classiques. Plus pr\'ecisement notre cadre est le suivant.
Soit $(M, g)$ une variété riemannienne lisse  compacte, sans bord connexe 
de dimension $d$ et
$\mathbf{\Delta}$ le laplacien sur $(M,g)$. 
Soit $0=\omega_{0}<\omega_{1}\leq \omega_{2}\leq ...$ le spectre de  
$\sqrt{-\mathbf{\Delta}}$ et $(e_j)_{j\geq 0}$ 
 une base orthonormale $L^2$ de fonctions propres réelles, de sorte que
$-\mathbf{\Delta} e_j=\omega_j^2e_j$. 
Soient $0<a<b$ et  
$E_{h}$ le sous espace de $L^2(M)$
\begin{equation}\label{2.1}
E_h=\{u=\sum_{k\in I_h}z_ke_k(x), \  z_k\in \mathbb C\},  \quad I_h=\{k, \  h\omega_k\in ]a, b]\}.
\end{equation}
Soit $N_{h}= \dim(E_{h})$. D'apr\`es la formule de Weyl, avec reste pr\'ecis\'e (voir \cite[Theorem 1.1]{H68}), on a 
pour $h\in ]0,1]$
\begin{equation}\label{2.1bis}
N_{h}=(2\pi h)^{-d}\text{Vol}(M)\text{Vol}(S^{d-1})\int_{(a,b)} \rho^{d-1}d\rho +O(h^{-d+1}).
\end{equation}

Rappelons qu'il existe une constante $C$ ind\'ependante de $h\in ]0,1]$ telle 
qu'on a 
\begin{equation}\label{0.8}
\Vert u \Vert_{L^\infty(M)} \leq C h^{-d/2}\Vert u \Vert_{L^2(M)}
\quad \forall u\in E_{h}
\end{equation}
et que plus g\'en\'eralement, si $A(x,hD_x)$ est un opérateur $h$-pseudodifférentiel classique
sur $M$ de degr\'e $0$ et à support essentiel contenu dans $\{(x,\xi)\in T^*M, \  
\vert\xi\vert_x \leq L\}$ pour un $L<\infty$, il existe une constante $C$ ind\'ependante de $h\in ]0,1]$ telle que 
pour tout $1\leq p\leq r \leq\infty$, on~a 
\begin{equation}\label{0.8bis}
 \Vert A(x,hD_x)g \Vert_{L^r(M)} \leq C h^{-d({\frac 1 p}-{\frac 1 r})}\Vert g \Vert_{L^p(M)} \quad \forall g\in L^p(M).
 \end{equation}

\bigskip
Les in\'egalit\'es de Sobolev \eqref{0.8} ou \eqref{0.8bis} sont optimales.
L'objectif de cet article est d'\'etudier des versions probabilistes de ces 
in\'egalit\'es. D\'ecrivons rapidement le type de r\'esultats que nous obtenons: 

On note $S_{h}$ (resp. $\widetilde{S}_{h}$) la sph\`ere unit\'e de l'espace euclidien
$E_{h}=\mathbb C^{N_h}$ (resp. $\widetilde {E}_{h}=\mathbb R^{N_h}$), et $P_{h}$ (resp. $\widetilde{P}_{h}$) la probabilit\'e uniforme sur $S_{h}$ (resp. $\widetilde{S}_{h}$). On verra dans la section \ref{sec2} (voir en particulier le th\'eor\`eme \ref{thm2.1}) que les probabilit\'es $P_{h}$ et $\widetilde{P}_{h}$ sont associ\'ees \`a une r\'epartition uniforme de l'\'energie dans l'espace de 
phase $T^*M$ pour la mesure de Liouville canonique $d\lambda$ sur $T^*M$.
On notera $\mathbb E_{h}( f)=\int_{S_{h}}f(u)dP_{h}$
l'esp\'erance d'une variable al\'eatoire $f$, et $\Pi_{h}$ le projecteur orthogonal de $L^2(M)$
sur $E_{h}$.

On a alors le résultat probabiliste 
essentiellement classique suivant, qui estime la mesure des
$u\in E_{h}$ de grande norme $L^{\infty}$.

\begin{thm}\label{thm0.1}
Pour tout $c_2<\text{Vol}(M)$, il existe $C>0$ tel que pour tout $h\in ]0,1]$ et tout $\Lambda\geq 1$ on ait, avec $c_1=d(1+d/2)$

\begin{equation}\label{0.11}\begin{aligned}
P_{h}\big(u\in S_{h}, \   \Vert u \Vert_{L^\infty} > \Lambda\big)
&\leq Ch^{-c_1}e^{-c_2\Lambda^2},\\
\widetilde{P}_{h}\big(u\in \widetilde{S}_{h}, \   \Vert u \Vert_{L^\infty} > \Lambda\big)
&\leq Ch^{-c_1}e^{-c_2\Lambda^2}.
\end{aligned}
\end{equation}
\end{thm}

Nous donnerons une preuve du th\'eor\`eme \ref{thm0.1} 
dans la section \ref{sec2}. L'estimation \eqref{0.11}
a une cons\'equence imm\'ediate sur 
les versions probabilistes des injections de Sobolev.
Rappelons que pour $p\in [1,\infty]$ et $s\geq 0$,  
l'espace de Sobolev $W^{s,p}$ est d\'efini par 
\begin{equation}\label{0.1}
W^{s,p}=\{f\in L^p(M), \  (1-\mathbf{\Delta})^{s/2} f \in L^p(M)\}.
\end{equation}
Les espaces $W^{s,p}$ sont indépendants du choix de la métrique $g$ sur M,
et pour $ 1\leq p \leq r <\infty$, on a les injections de  Sobolev 
\begin{equation}\label{0.2}
W^{s,p} \subset L^r, \quad  s=\frac{d}{p}-\frac{d}{r}.
\end{equation}
Rappelons aussi la construction de Littlewood-Paley. 
On fixe $0<a<c$, et $\varphi \in C^\infty(\mathbb R)$ tel que $\varphi(t)=0$ pour $t\leq a$,
$\varphi(t)=1$ pour $t\geq c$, et $\varphi'(t)>0$ pour $t\in ]a,c[$. On pose
$\psi_{-1}(t)=1-\varphi(t), \  \psi(t)=\varphi(t)-\varphi(t/2), \  \psi_n(t)=\psi(2^{-n}t)$ pour $n\geq 0$. Alors
$\psi$ est à support dans $[a,2c]$, $\psi(t)>0$ pour $t\in ]a,2c[$ 
et $1=\sum_{n\geq -1}\psi_n(t)$ pour tout $t$. 
Posons $b=2c>a>0$. Pour toute distribution $f\in \mathcal D'(M)$, on a $f=\sum c_k(f) e_k$, où les 
$c_k(f)=\int_M fe_k dx$ sont les coefficients de Fourier de $f$, et
la décomposition de Littlewood-Paley de f s'écrit, avec $h_k = 2^{-k}$, 
\begin{equation}\label{0.5}
f=\sum_{n=-1}^\infty f_n, \quad f_n= \psi_n(\sqrt{\vert\mathbf{\Delta}\vert})f=
\sum_k \psi_n(\omega_k) c_k(f) e_k ,  \  f_n \in E_{h_n} \  (n\geq 0 ).
\end{equation}
Rappelons que pour $q,r\in [1,\infty]$ et $s\in \mathbb R$, l'espace de Besov
$B^s_{q,r}$ est l'espace des distributions $f\in \mathcal D'(M)$ dont la décomposition de
 Littlewood-Paley vérifie
\begin{equation}\label{0.6}
  \text{la suite} \ n\rightarrow 2^{ns}\Vert f_n\Vert_{L^q(M)} \
\text{appartient à} \  l^r(\mathbb N).
\end{equation}
Les éléments de $E_{h_n}$ sont des fonctions à échelle $h_n=2^{-n}$ sur $M$, et les injections de Sobolev 
\eqref{0.2} peuvent \^etre vues comme conséquence des inégalités \eqref{0.8bis}. 

\noindent
Soit alors $X$ l'espace produit
\begin{equation}\label{0.9}
X=\Pi_{n=0}^\infty S_{h_n}.
\end{equation}
On munit $X$ de la probabilité produit $\mathbb P=\Pi_{n=0}^\infty P_{h_n}$. 
Soit $(a_{n})_{n\geq 0}$ une suite de r\'eels positifs telle que $\sum_{n}n^{1/2}a_{n} <\infty$.
Soit  $j$
l'application  de $X$ dans l'espace de Besov $B^0_{2,\infty} $
\begin{equation}\label{0.10}
\ba
X & \rightarrow B^0_{2,\infty} \\
g=(g_n)_{n\geq 0}&\mapsto j(g)=
\sum_{n=0}^\infty a_{n}g_n.
\ea
\end{equation}
Comme  corollaire immédiat du théorème \ref{thm0.1}, on obtient 

\begin{cor}\label{cor0.1}
On a $\mathbb P( j(g)\in C^0(M))=1$.
\end{cor}

\begin{remk} On notera que le corollaire \ref{cor0.1} est violent, puisqu'il implique 
en particulier une injection 
presque sure de $B^\sigma_{2,\infty}$ dans $C^0(M)$ pour tout $\sigma>0$, soit un 
gain de $d/2$ dérivées par rapport à l'injection de Sobolev.  
\end{remk}
\begin{proof}
Soit $A>0$ donn\'e, $m_n=(An\log(2))^{1/2}$ et $B_n$ la partie de $S_{2^{-n}}$
$$B_n=\{g_n, \  \Vert g_n \Vert_{ L^\infty}\leq m_n\}$$
D'après \eqref{0.11} on a 
\begin{equation}\label{0.13}
P_{h_n}( B_n) \geq 1-C 2^{-n(c_{2}A-c_1)}
\end{equation} 
Soit $B$ la partie de $X$, $B=S_{h_0}\times\Pi _{n=1}^\infty B_n$.
Pour $g=(g_n) \in B$ et
$f=j(g)$, on a
\begin{equation}\label{0.14}
\Vert f \Vert_{ L^\infty }\leq \sum_{n=0}^\infty a_{n}\Vert g_n \Vert_{ L^\infty}
\leq Ca_{0}+ (A\log(2))^{1/2}\sum_{n=1}^\infty n^{1/2}a_{n}.
\end{equation} 
On a alors pour tout $f\in j(B)$,
$f\in C^0(M)$ d'après \eqref{0.14}, puisque les $g_{n}$ sont continus, et d'après \eqref{0.13}
$$\mathbb P(B)=\prod_{n=1}^\infty P_{h_n}(B_n) \geq  \prod_{n=1}^\infty 
\Big(1-C2^{-n(c_{2}A-c_1)}\Big)\geq 1-\varepsilon,$$
avec $\varepsilon >0$ petit si  la constante $A$ est  
assez grande, d'où le résultat.
\end{proof}

La morale du corollaire \ref{cor0.1} est la suivante : 
si on se donne une famille de fonctions $g_h\in  E_{h}$
à échelle $h$ et 
d'énergie $1$ pour tout $h=2^{-n}$, et si on re-répartit leur énergie aléatoirement dans l'espace de phase,
on obtient une nouvelle famille de fonctions dans $E_{h}$ 
qui est  "presque" bornée dans le sens où 
$sup_h \Vert g_h\Vert_{L^\infty} \vert \log (h)\vert^{-1/2}$ l'est .

Nos constructions de mesures sur l'espace $L^2(M)$ (voir l'appendice C  pour la 
construction pr\'ecise) utilisent une 
d\'ecomposition orthogonale $L^2(M)=\oplus_{k} E_{k}$, o\`u les $E_{k}$ sont des sous-espaces 
de dimensions finis invariants 
par l'op\'erateur $\mathbf {\Delta}$. On choisit en particulier sur 
 chaque $E_{k}$ une probabilit\'e $P_{k}$
invariante par les isom\'etries de $E_{k}$, et on munit l'espace $L^2$ de la probabilit\'e
produit $P=\Pi_{k}P_{k}$. Dans notre cadre, si $\omega$ est la fr\'equence typique
des \'el\'ements de $E_{k}$, on a toujours $C_{1}\omega^{d-1}\leq \dim (E_{k})\leq C_{2}\omega^{d}$, et plus pr\'ecisemment, les fr\'equences $\omega$ des \'el\'ements
de $E_{k}$ v\'erifient $\omega \in (a_{k},b_{k}), a_{k}+ C\leq b_{k}\leq ca_{k}$,
avec $C>0,c>1$.
Le fait de choisir des espaces $E_{k}$ de "grande dimension" permet d'obtenir des
r\'esultats plus fort avec probabilit\'e $1$ que le choix  $E_{k}=\mathbb C e_{k}$,
qui v\'erifie  $\dim (E_{k})=1$, et pour lequel nous renvoyons aux travaux de N. Tzvetkov \cite{AT},
\cite{T1} et \cite{T2}. De plus, on verra dans la section \ref{sec2} comment  le choix que
nous faisons des $E_{k}$
permet de relier naturellement nos probabilit\'es \`a la mesure de Liouville sur $T^*M$.

\bigskip
 L'article est organis\'e comme suit. Dans la section~\ref{sec2} nous d\'emontrons le th\'eor\`eme~\ref{thm0.1} et des versions pr\'ecis\'ees, en autorisant  des localisations spectrales plus fines que (\ref{2.1}). Le th\'eor\`eme \ref{thm2.1} de la section \ref{sec2.2} pr\'ecise
 le fait que nos mesures sont associ\'ees \`a la mesure de Liouville sur $T^*M$.
 Dans la section \ref{sec2.3} nous d\'ecrivons pour $2<q\leq\infty$ les estimations $L^q$
  presques sures. On trouvera dans Shiffman-Zelditch~\cite{ShZe} des preuves  analogues 
 pour les  estim\'ees sur
les sections de fibr\'es holomorphes. Les bornes inf\'erieures que nous obtenons sur les
m\'edianes des normes $L^q$ semblent nouvelles.  Dans les sections suivantes, nous donnons des applications simples \`a l'\'etude de solutions d'\'equations aux d\'eriv\'ees partielles. Notre premi\`ere application (dans la section~\ref{sec3}) concerne la croissance des normes $L^p$ des harmoniques sph\'eriques (les fonctions propres du Laplacien sur les spheres $\mathbb{S}^d\subset \mathbb{R}^{d+1}$). Il est connu depuis les travaux de H\"ormander~\cite{H68} et de Sogge~\cite{So88} que sur toute vari\'et\'e riemannienne compacte de dimension $d$, $(M, g)$, les fonctions propres du Laplacien v\'erifient les estimations suivantes
 \begin{thm*}Pour tout $2\leq p \leq + \infty$, il existe $C>0$ tel que  pour toutes fonctions propres du laplacien, 
  $u$,  $-\mathbf{\Delta}_g u = \lambda ^2 u$, on a 
 \begin{equation}\label{eq.Lp}
 \|u\|_{L^p(M)}\leq C \lambda^{\delta (p)}\|u\|_{L^2(M)},
 \end{equation}
 avec 
 \begin{equation}
 \delta (p) = \begin{cases}
&\frac{ (d-1)} 2 - \frac d p \text{ si } p \geq \frac{ 2(d+1)} {d-1}\\
&\frac{ (d-1)} 2 \bigl( \frac 1 2 - \frac 1 p \bigr) \text{ si } p \leq \frac{ 2(d+1)} {d-1}.
\end{cases}
\end{equation}
\end{thm*}
On sait par ailleurs que ces estim\'ees sont optimales sur les sph\`eres (munies de leurs m\'etriques standart). Dans le premier r\'egime, les harmoniques sph\'eriques zonales (qui se concentrent en deux points diam\'etralement oppos\'es) r\'ealisent l'optimum tandis que dans le second, ce sont les harmoniques qui se concentrent sur un \'equateur qui saturent les estim\'ees~\eqref{eq.Lp}. Notre premi\`ere application (voir Th\'eor\`eme~\ref{thm4}) montre que si on choisit au hasard, pour la mesure de probabilit\'e naturelle (voir section~\ref{sec3}) une base Hilbertienne de $L^2( \mathbb{S}^d)$ form\'ee d'harmoniques sph\'eriques, alors avec probabilit\'e $1$, pour tout $p<+\infty$, toutes les normes $L^p$ sont born\'ees (uniform\'ement). Autrement dit, on peut prendre $\delta (p) =0$ dans~\eqref{eq.Lp} avec probabilit\'e $1$.  On remarquera que ce ph\'enom\`ene d'existence de familles de fonctions propres exhibant des comportements diff\'erents en ce qui concerne la croissance des normes $L^p$ n'est pas si surprenant puisqu'il se manifeste aussi sur les tores $\mathbb{T}^d$. En effet, dans ce cadre la situation est renvers\'ee puisque les fonctions propres naturelles ($e^{i n\cdot x}, n \in \mathbb{Z}^d$) ont toutes leurs normes $L^p$ born\'ees. Cependant, il est possible de d\'emontrer (voir la section~\ref{sec.3.2}) qu'il existe sur $\mathbb{T}^d$ une suite de fonctions propres du Laplacien $u_n$ v\'erifiant
\begin{equation}\label{eq.mino-tores}
 \| u_n \|_{L^p( \mathbb{T}^d)} \geq |\lambda_n| ^{\frac {d-2} 2 - \frac d p},
 \end{equation}
 et donc pour $d\geq 3$, $p \geq 2d/ (d-2)$, les normes $L^p$ ne sont pas uniform\'ement born\'ees (voir~\cite{Bo} pour des majorations sur $\mathbb{T}^d$) .
  
 Notre deuxi\`eme application (section~\ref{sec4}) concerne l'\'etude de l'\'equation des ondes amorties sur une vari\'et\'e compacte. On consid\`ere donc pour $a\in C^\infty(M; [0,+\infty[)$ les solutions de 
 $$ (\partial_t^2 - \mathbf{\Delta}) u + a(x) \partial _t u =0, \qquad (u, \partial_t u) \mid_{t=0} = (u_0, u_1) \in H^1(M) \times L^2(M).$$
 Leur \'energie 
 $$ \mathcal{E}(u) (t) = \frac 1 2\int_M (|\nabla_x u|^2 + |\partial_tu|^2 )dx $$ v\'erifie 
 $$ \frac {d \mathcal{E}(t)} {dt} = - \int_M a(x) |\partial_t u|^2 dx,$$ et est donc une fonction d\'ecroissante dont on peut d\'emontrer qu'elle tend vers $0$ quand $t$ tend vers l'infini d\`es que l'amortissement $a$ est non trivial. Si de plus il existe un taux de d\'ecroissance uniforme par rapport \`a l'\'energie initiale,  la propri\'et\'e de semi-groupe montre que ce taux est alors toujours exponentiel:
 $$\mathcal{E}(u)(t) \leq C e^{-ct} \mathcal{E}(u) (0).$$
 \begin{thm*}[Bardos-Lebeau-Rauch~\cite{BaLeRa92}]
 Il existe un taux de d\'ecroissance (exponentiel) uniforme si et seulement si toutes les g\'eod\'esiques de la vari\'et\'e $M$ rencontrent la r\'egion $a>0$.
 \end{thm*}
 Ici, on s'int\'eresse \`a des situations o\`u cette propri\'et\'e g\'eom\'etrique n'est plus v\'erifi\'ee, mais o\`u elle est viol\'ee pour \og un ensemble rare de g\'eod\'esiques\fg. Plus pr\'ecisement, si on appelle $\mathcal{E}$ l'ensemble des points de l'espace des phases tels que le long de la g\'eod\'esique issue de ce point, la moyenne asymptotique de l'amortissement est nulle, alors la mesure dans l'espace des phases de $\mathcal{E}$ est nulle. En particulier, si le flot est ergodique, cette propri\'et\'e est v\'erifi\'ee. Nous d\'emontrons alors que pour des mesures de probabilit\'es naturelles sur l'espace d'\'energie $H^1\times L^2$, il existe toujours un taux uniforme de d\'ecroissance, sur des ensembles de mesures arbitrairement proches de $1$. 
 
 Finalement, notre derni\`ere application (section~\ref{sec5}) concerne la th\'eorie de Cauchy pour l'\'equation des ondes semilin\'eaire sur une vari\'et\'e compacte de dimension $3$.
 \begin{equation}
(\partial_t^2 - \mathbf{\Delta} ) u + u^{p} =0 , (u\mid_{t=0}, \partial_t u \mid_{t=0}) = (u_0, u_1) \in (H^1(M)\cap L^{p+1}(M))\times L^2 (M),
\end{equation}
o\`u $p$ est un entier impair. On connait pour ce syst\`eme l'existence de solutions faibles globales en temps. De plus, pour $p\leq 5$, ces solutions sont fortes et uniques. Nous d\'emontrons que pour une famille de mesures de probabilit\'es naturelles sur l'espace $H^1(M) \times L^2(M)$, il existe pour presque toute donn\'ee initiale $(u_0, u_1)$ et  tout $p <+\infty$ une solution locale forte (en un sens qui sera pr\'ecis\'e) et que sur l'intervalle d'existence de ces solutions fortes, il y a unicit\'e des solutions faibles (i.e. toute solution faible co\"{\i}ncide avec cette solution forte).

Certains de nos r\'esultats restent vrais sur une vari\'et\'e \`a bord. Dans une derni\`ere section, nous donnons les \'el\'ements permettant dans ce cadre d'adapter les d\'emonstrations. 
Finalement, nous avons rassembl\'e dans un appendice quelques r\'esultats de calcul des probabilit\'es et de calcul pseudo-diff\'erentiel n\'ecessaires \`a la compr\'ehension de l'article.

Ce projet a b\'en\'efici\'e du soutien de l'Agence Nationale de
la Recherche, projet ANR-07-BLAN-0250

\section {Estimations probabilistes}\label{sec2}

Dans cette section, nous calculons les lois de certaines variables al\'eatoires associ\'ees 
\`a la th\'eorie de Littlewood-Paley sur la vari\'et\'e $M$. Les asymptotiques de Weyl jouent un r\^ole cl\'e dans ces calculs. Nos r\'esultats autorisent des localisations en fr\'equence plus fins et plus g\'en\'eraux que les localisations dyadiques de l'introduction. Plus pr\'ecisement, on consid\`erera  $0<a_h<b_h\leq c$ deux fonctions d\'efinies pour $h\in (0, h_0)$ telles que
\begin{equation}
\lim_{h\rightarrow 0} b_h= b \geq \lim_{h\rightarrow 0}a_h =a \geq 0.
\end{equation}
On supposera que si $a=b$, alors $a>0$ et
\begin{equation}\label{eq.mino}
b_h - a_ h \geq Dh
\end{equation}
pour une constante $D$ assez grande (\`a pr\'eciser ult\'erieurement)
On notera $E_{h}$ (resp $\widetilde{E}_{h}$) le sous espace de $L^2(M)$
\begin{equation}\label{2.1.1}
\begin{gathered}
E_{h}=\{u=\sum_{k\in I_{h}}z_ke_k(x), \  z_k\in \mathbb C\},  \quad \widetilde{E}_{h}=\{u=\sum_{k\in I_{h}}z_ke_k(x), \  z_k\in \mathbb R\},\\
 I_h= \{ k \in \mathbb{N}; h\omega_k \in ]a_h, b_h]\}.
\end{gathered}\end{equation}
Soit $N_h=dim(E_{h})$.
Rappelons que d'apr\`es la formule de Weyl, avec reste pr\'ecis\'e~\eqref{2.1bis} (voir H\"ormander \cite{H68}) , on a 
pour $h\in ]0,1]$, avec $c_d=\text{Vol} (x\in \mathbb R^d, \vert x\vert \leq 1)$ le volume de la boule unit\'e  en dimension $d$,
\begin{equation}\label{2.1.1bis}
\exists C>0; \forall \lambda >0
\Bigl| \sharp \{ k\in \mathbb{N}; \omega_k \leq \lambda\} - c_d\frac{ \text{ Vol }(M)}{(2\pi)^d}  \lambda ^d \Bigr|\leq \lambda^{d-1},
\end{equation}
et donc
\begin{equation}\label{eq.est0}
\Bigl| \sharp \{ k; \omega_k \in I_{h} \} -c_d\frac{ \text{ Vol }(M)}{(2\pi)^d}  \Bigl((h^{-1}b_h) ^{d}- (h^{-1}a_h) ^{d}\Bigr)\Bigr| \leq  C h^{-d+1},
\end{equation}
\begin{multline}\label{eq.mino1}
N_h = \sharp \{ k; \omega_k \in I_{h} \} \sim \\
\begin{cases} c_d\frac{ \text{ Vol }(M)}{(2\pi)^d} (b_h^d- a_h^d) h^{-d}, &\text{ si } 0\leq a <b,\\
c_d\frac{ \text{ Vol }(M)}{(2\pi)^d}d a^{d-1}h^{-d } \bigl[(b_h - a_h )+\mathcal{O} (h+(b_h -a_h )^2+(b_h - a_h)|a- a_h|)\bigr], &\text{ si } 0<a=b.
\end{cases}
\end{multline}
On en d\'eduit 
\begin{equation}\label{eq.mino2}
\begin{gathered}
\exists D_0>0; b_h - a_ h \geq D_0h \Rightarrow \exists \beta>\alpha >0; \\
 \alpha h^{-d}(b_h - a_h) \leq N_h = \sharp \{ k; \omega_k \in I_{h} \} \leq \beta h^{-d}(b_h - a_h) .
 \end{gathered}
\end{equation}
Nous supposerons dans la suite que la constante $D$ dans ~\eqref{eq.mino} est choisie de telle fa\c con que ~\eqref{eq.mino2} est v\'erifi\'ee, et qu'on ait aussi 
pour tout $h$ 
\be\label{minordim}
N_{h}\geq 2
\ee
On munit les sph\`eres unit\'e de $E_h$ (resp. $\widetilde{E}_h$) de la mesure de probabilit\'e uniforme, $P_h$ (resp. $\widetilde{P}_h$).  
\subsection{Estimations $L^\infty$ presque sures}\label{sec2.1}
Dans le cadre que nous venons de d\'evelopper, le th\'eor\`eme~\ref{thm0.1} est un cas particulier de
\begin{thm}\label{thm0.2}
Il existe $C>0,c_{2}>0$ tel que pour tout $h\in ]0,1]$ et tout $\Lambda\geq 1$ on ait, avec $c_1=d(1+d/2)$
\begin{equation}\label{0.11bis}\begin{aligned}
P_{h}\big(u\in S_{h}, \   \Vert u \Vert_{L^\infty} > \Lambda\big)
&\leq Ch^{-c_1}e^{-c_2\Lambda^2}\\
\widetilde{P}_{h}\big(u\in \widetilde{S}_{h}, \   \Vert u \Vert_{L^\infty} > \Lambda\big)
&\leq Ch^{-c_1}e^{-c_2\Lambda^2}.
\end{aligned}
\end{equation}
De plus, on peut choisir  $c_2\in ]0,\text{Vol}(M)[$ dans le cas 
$\lim_{h\rightarrow 0} (b_{h}-a_{h})/h=+\infty$.
\end{thm}
On se limitera dans la preuve au cas complexe, le cas r\'eel \'etant similaire.
 Pour tout $x\in M$,  et tout $\lambda \in \mathbb{R}^+$, on note 
 \begin{equation}
 \begin{gathered}
  E_{x,\lambda}= \sum_{k; \omega_k \leq \lambda} |e_k (x)|^2, \qquad 
  e_{x,h}=E_{x,h^{-1}b_h} - E_{x,h^{-1}a_h}\\
  b_{x,h}= ( e_k(x))_{k\in I_{h}} \in \mathbb C^{N_h}.
 \end{gathered}
 \end{equation}
On a $e_{x,h}=\vert b_{x,h}\vert^2$. Soit $ev_x$ la variable al\'eatoire sur $(S_{h},P_{h})$
\begin{equation}\label{2.2}
ev_x(u)=u(x)=\sum_{k\in I_{h}}a_ke_k(x)=(a\vert b_{x,h})=(a\vert {b_{x,h}
\over \vert b_{x,h}\vert})
\vert b_{x,h}\vert.
\end{equation}
Le vecteur $\varepsilon_x={b_{x,h}\over \vert b_{x,h}\vert}$ est de norme $1$ et on a
pour tout $r\geq 0$, avec $P=P_{h}$

\begin{equation}\label{2.3}
P\Big( \vert ev_x\vert >r\Big)=
P\Big( \vert (a\vert \varepsilon_x)\vert >r/{\vert b_{x,h}\vert}\Big)=\Phi(r/{\vert b_{x,h}\vert}),
\end{equation}
avec $\Phi(t)=P( \vert (a\vert \varepsilon)\vert >t)$ o\`u $\varepsilon$
est un vecteur unitaire quelconque.
D'apr\`es \eqref{3.3bis} (on identifie ici la sph\`ere unit\'e de $\mathbb{C}^{N_h}$ avec celle de $\mathbb{R}^{2N_h}$), on a
\begin{equation}\label{2.4}
\Phi(t)={\bf 1}_{t\in [0,1[}(1-t^2)^{N_h-1}.
\end{equation}

\begin{lem}\label{lem2.1} Il existe  $C_0>0$ 
tel que pour tout $x\in M$ et tout $h\in ]0,1]$ on a
\begin{equation}\label{2.5}
\vert e_{x,h}-{N_{h}\over \text{Vol}(M)}\vert \leq C_{0}h^{-d+1}
\end{equation}
En particulier, si $a_h, b_h$ v\'erifient 
$b_{h}-a_{h}\geq Dh$ avec $D$ assez grand, d'apr\`es (\ref{eq.mino1}) et (\ref{eq.mino2}), on a avec $C$
ind\'ependant de $x\in M$ et $h\in ]0,1]$ 
\be\label{2.5bis}
N_{h}/C \leq e_{x,h}\leq CN_{h}.
\ee
\end{lem}
\begin{proof}D'apr\`es~\cite[Th\'eor\`eme 1.1]{H68}, on a avec $C$ ind\'ependant de 
$x$ et de $\lambda$
$$ \Bigl |E_{x,\lambda}-\frac{ c_{d}\lambda^d} {(2\pi)^d} \Bigr| \leq C \lambda^{d-1},
$$ et donc~\eqref{2.5} est cons\'equence de~\eqref{eq.est0}. 
 \end{proof}
On en d\'eduit en particulier le lemme suivant 
 \begin{lem}\label{lem.est}
Il existe $c_{2}>0$ tel que pour tout $x\in M$ et tout $\lambda >0$, 
\begin{equation}\label{eq.point}
P_h ( u\in S_h; | u(x)| >\lambda ) \leq e^{-c_{2} \lambda^2}.
\end{equation} 
\end{lem}
On effet, d'apr\`es~\eqref{2.3} et~\eqref{2.4}, 
\begin{equation}
P_h ( u\in S_h; | u(x)| >\lambda ) \leq (1- \frac{\lambda^2} {|b_{x,h}|^2} )^{N_h-1}
\leq e^{-\frac{N_h-1 } {|b_{x,h}|^2} \lambda^2}=e^{-\frac{N_h-1 } {e_{x,h}} \lambda^2}
\end{equation} 
et le lemme~\ref{lem.est} s'en d\'eduit d'apr\`es le lemme~\ref{lem2.1}. 

\bigskip

{\bf Preuve du th\'eor\`eme \ref{thm0.2}}. On se limite encore ici au cas complexe (le cas r\'eel \'etant similaire).
Le th\'eor\`eme \ref{thm0.2} est une cons\'equence  des formules 
\eqref{2.3}, \eqref{2.4}
et de l'estimation \eqref{2.5}.   
Remarquons qu'il existe $0<c<C$ et $C_1>0$ 
tels que pour tout $h\in ]0,1]$ et tout $u\in S_{h}$
 on a $c \leq \Vert u \Vert_{L^\infty}\leq Ch^{-d/2}$ et
$$\Vert \nabla_x u \Vert_{L^\infty} \leq C_1h^{-(d/2+1)}.$$
Il en résulte
\begin{equation}\label{4.4}
sup_x \  \vert u(x)\vert  \leq
sup_{\alpha\in A} \  \vert u(x_\alpha)\vert
+  \varepsilon \Lambda 
\end{equation}  
dès que $x_\alpha, \alpha\in A$ est un réseau de points de $M$ de maille plus petite
que $\varepsilon \Lambda h^{d/2+1}/C_{1}$. 
D'apr\`es \eqref{2.3}, \eqref{2.4}, 
le lemme (\ref{lem2.1}), et
$(1-t^2)^{N_h}\leq e^{-{N_h}t^2}$ pour $t\in [0,1]$, il existe $h_0>0$ tel que pour 
 $h\in ]0,h_0]$ on ait
\begin{equation}\label{4.5}
\ba
P\big( u\in S_h, \  \Vert u \Vert_{L^\infty}>  \Lambda\big)& 
\leq \sum_{\alpha\in A}
P\Big(  \vert ev_{x_\alpha}\vert > (1-\varepsilon)\Lambda\Big)\\
&\leq \sum_{\alpha\in A}
{\bf 1}_{(1-\varepsilon)\Lambda\leq \vert b_{x_{\alpha},h}\vert}
(1-(1-\varepsilon)^2\Lambda^2/{\vert b_{x_{\alpha},h}\vert^2})^{N_h-1}\\
&\leq \sum_{\alpha\in A}e^{-(N_h-1)
(1-\varepsilon)^2\Lambda^2/{\vert b_{x_{\alpha},h}\vert^2}}\\
&\leq card(A)e^{-Vol(M)(1-\varepsilon)^2\Lambda^2 {N_{h}-1\over N_{h}+C_{0}Vol(M)h^{-d+1}}}.
\ea
\end{equation}
On a $\liminf_{h\rightarrow 0}{N_{h}-1\over N_{h}+C_{0}Vol(M)h^{-d+1}}>0$ d'apr\`es
(\ref{eq.mino2}), et $\lim_{h\rightarrow 0}{N_{h}-1\over N_{h}+C_{0}Vol(M)h^{-d+1}}=1$
d'apr\`es (\ref{eq.mino1}) dans le cas $\lim_{h\rightarrow 0}(b_{h}-a_{h})/h=+\infty$.
Comme on a $\sharp(A)\leq C\varepsilon^{-d}\Lambda^{-d}h^{-d(d/2+1)}$,
\eqref{0.11bis} résulte de \eqref{4.5}.
La preuve du th\'eor\`eme (\ref{thm0.2}) est complète.
\subsection{Mesures de d\'efaut presque sures}\label{sec2.2}
On se place dans cette section sous l'hypoth\`ese
\begin{equation}\label{eq.hyp}
\lim_{h\rightarrow 0} \frac{ b_h - a_h} h = + \infty,
\end{equation}
ce qui exclut le cas critique $b_h - a_h \sim Dh$.
On a alors d'apr\`es (\ref{eq.mino1})
$$\lim_{h\rightarrow 0}{h^{-d+1}\over N_{h}} =0 $$
\begin{lem}\label{lem2.2.1}
Soit $A(x,hD)$ un op\'erateur h-pseudodiff\'erentiel classique sur $M$
de degr\'e $0$ et de symbole principal $a(x,\xi)$.
Il existe $C_{0}>0$  tel que 
\begin{equation}\label{2.2.1}
\vert \mathbb E_{h}\Big( (A(x,hD)u\vert u) \Big)-
{\int_{\vert \xi\vert_x\in I} a(x,\xi)d\lambda\over 
\int_{\vert \xi\vert_x\in I}d\lambda}\vert \leq C_{0}{h^{-d+1}\over N_{h}}
\end{equation} 
o\`u $d\lambda=dxd\xi$ est la mesure de Liouville, $I= ]a,b[$ si $a<b$, et si $a=b >0$ on a not\'e 
$${\int_{\vert \xi\vert_x\in I} a(x,\xi)d\lambda\over 
\int_{\vert \xi\vert_x\in I}d\lambda}= \lim_{\epsilon \rightarrow 0} {\int_{\vert \xi\vert_x\in ]a-\epsilon,a+\epsilon[} a(x,\xi)d\lambda\over 
\int_{\vert \xi\vert_x\in ]a-\epsilon,a+\epsilon[}d\lambda}.
$$
\end{lem}
 \begin{proof} Pour tout op\'erateur lin\'eaire $A$ sur $E_{h}$ on a 
$$\mathbb E_{h}\Big( (A u\vert u) \Big)= {tr(A)\over N_{h}}$$
et il suffit donc de v\'erifier qu'on a\begin{equation}\label{2.2.2}
\vert \frac{ tr(\Pi_{h}A(x,hD)\Pi_{h})}{N_h}-m_{a}\vert \leq C_{0}{h^{-d+1}\over N_{h}}
\end{equation}
o\`u on a not\'e
\be\label{defma}
m_{a}={\int_{\vert \xi\vert_x\in I} a(x,\xi)d\lambda\over 
\int_{\vert \xi\vert_x\in I}d\lambda}
\ee

Ce r\'esultat est cons\'equence de travaux de Guillemin~\cite{Gu}. Ici, nous utiliserons la preuve donn\'ee par H\"ormander. On a d'apr\`es la preuve de~\cite[Theorem 29.1.7]{Ho85-1} (voir les pages 259-260)
\begin{prop}
Notons 
$E_\lambda= 1_{\sqrt{- \Delta} <\lambda}$ le projecteur spectral. Alors pour tout op\'erateur pseudodiff\'erentiel d'ordre $0$, $B$, de symbol principal $b(x,\xi)$, on a
\begin{equation}\label{eq.trtace}
 \text{Tr} ( E_\lambda B E_\lambda) = \iint_{|\xi|_x <\lambda} b(x, \xi) dx d\xi + O(\lambda^{d-1}).
\end{equation}
\end{prop}
On remarque ensuite que dans ~\eqref{eq.trace}, d'apr\`es la cyclicit\'e de la trace, on peut remplacer $E_\lambda B E_\lambda$ par $E_\lambda^2 B= E_\lambda B$. Par ailleurs,
(voir appendice \ref{sec5.2}, en particulier la formule (\ref{5.2.1})), $B=A(x,hD)$ est, uniform\'ement en $h\in ]0,1]$, un op\'erateur pseudodiff\'erentiel d'ordre $0$ au sens de la 
proposition pr\'ec\'edente. De plus, on peut r\'eecrire (\ref{eq.est0}) sous la forme
$$ N_{h}={h^{-d}\over (2\pi)^d} \iint_{a_{h}<|\xi|_x <b_{h}} dxd\xi +O(h^{-d+1}) $$
On a alors
\begin{multline}\label{eq.trace}
 \tr(\Pi_{h}A(x,hD)\Pi_{h})= \tr(E_{h^{-1} b_h}A(x, hD)) - \tr(E_{h^{-1} a_h}A(x, hD))\\
 = \frac 1{ (2\pi)^d}\iint_{h^{-1} a_h <|\xi|_x <h^{-1} b_h} a(x, h\xi) dx d\xi + O(h^{1-d})\\
 = \frac{ h^{-d} } {(2\pi)^d} \iint_{a_{h}<|\xi|_x <b_{h} } a(x, \xi) dx d\xi + O(h^{1-d})\\
 ={  \iint_{a_{h}<|\xi|_x <b_{h} } a(x, \xi) dx d\xi\over 
  \iint_{a_{h}<|\xi|_x <b_{h} }  dx d\xi}(N_{h}+O(h^{-d+1})) + O(h^{-d+1})
 \end{multline}
 On a aussi
 $$ \vert m_{a} - {  \iint_{a_{h}<|\xi|_x <b_{h} } a(x, \xi) dx d\xi\over 
  \iint_{a_{h}<|\xi|_x <b_{h} }  dx d\xi}\vert \leq Ch $$
ce qui, avec (\ref{eq.trace}) et $N_{h}\leq Ch^{-d}$ implique~\eqref{2.2.2}. La preuve du lemme \ref{lem2.2.1} est compl\`ete.
\end{proof}

Soit $(h_k)_{k\geq 0}$ une suite de limite nulle et $X$ l'espace topologique
produit $X=\Pi_k S_{h_k}$. On  munit $X$ de la probabilit\'e produit
$\mathbb P=\Pi_k P_{h_k}$.

\begin{thm}\label{thm2.1}
On suppose qu'il existe $L>0$ tel que $\sum_{k}h_{k}^L<\infty$ et \\
$\lim_{k\rightarrow \infty} N_{h_{k}}/\vert log(h_{k})\vert=\infty$ 
(cette derni\`ere hypoth\`ese est toujours satisfaite en dimension $d\geq 2$)
.
Alors on a, avec $m_{\sigma_{0}(A)}$ d\'efini par (\ref{defma}),
\begin{equation}\label{2.2.10}
\mathbb P\Big(\forall A \in \mathcal E^0_h, \ \lim_{k\rightarrow \infty} (A(x,h_kD)u_k\vert u_k)= m_{\sigma_{0}(A)}
\Big)=1
\end{equation}
\end{thm} 

\begin{remk}
Les hypoth\`eses sur la suite $h_{k}$ 
sont  \'evidemment satisfaites dans le cadre 
de la th\'eorie de Littlewood-Paley pour laquelle on a $h_k=2^{-k}$.
Le th\'eor\`eme \ref{thm2.1} dit que pour presque toute suite
$k\rightarrow u_k\in S_{h_k}$, la mesure de d\'efaut semi-classique de la suite  existe (i.e on n'a pas besoin d'extraire une sous suite) et est
toujours \'egale \`a la mesure de Liouville normalis\'ee sur la couronne
$\{(x,\xi)\in T^*M, \vert\xi\vert_x\in I\}$. 
\end{remk}

\begin{proof}Comme il existe une partie d\'enombrable $\mathcal D$ de $\mathcal E^0_h$
telle que pour tout $A\in \mathcal E^0_h$ et tout $\varepsilon>0$, il existe
$B\in \mathcal D$ tel que 
$$\limsup_{h\rightarrow 0}\Vert A(x,hD)-B(x,hD)\Vert_{L^2}+
\sup_{\vert \xi\vert_x\in I}\vert \sigma_0(A)(x,\xi)-\sigma_0(B)(x,\xi)\vert\leq \varepsilon$$ 
il suffit de prouver qu'on a pour tout 
$A\in \mathcal E^0_h$
\begin{equation}\label{2.2.11}
\mathbb P\Big(\lim_{k\rightarrow \infty} (A(x,h_kD)u_k\vert u_k)=m_{\sigma_{0}(A)}
\Big)=1
\end{equation}
En \'ecrivant  $A=A_1+iA_2$ avec $A_i$ auto-adjoint, on a 
$\sigma_0(A)=\sigma_0(A_1)+i\sigma_0(A_2)$ et on peut donc aussi supposer
que $A$ est auto-adjoint.
Notons $f_h$ la variable al\'eatoire r\'eelle sur $S_{h}$, $f_h(u)=(A(x,hD)u\vert u)$, $dp_h$ sa loi, qui est \`a support dans $[-a,a]$ avec 
$a=\sup_h \Vert A(x,hD)\Vert_{L^2}<\infty$. Soit $\mathcal M_h$ la m\'ediane de $f_h$.
La fonction $f_h$ est lipschitzienne
de constante de lipschitz $\leq 2\sup_h\Vert A(x,hD) \Vert_{L^2}$. D'apr\`es le th\'eor\`eme
de concentration de la mesure de P. Levy (voir la proposition  \ref{concentration} de l'appendice A. On remarquera que la sphere complexe $S_{h}$ s'identifie avec la sph\`ere r\'eelle $S^{2N_{h}-1}$, et que les constantes sont donc coh\'erentes avec~\eqref{eq.conc}) 
\begin{equation}\label{2.2.12}
P_{h}(u\in S_{h}, \ \vert f_h(u)-\mathcal M_h\vert\geq r)\leq 2\exp(-{(N_{h}-1)r^2
\over \Vert f_h\Vert^2_{lips}})
\end{equation}
On a
$\mathbb E (f_{h})=\mathcal M_h+\int_{-a}^{a}(x-\mathcal M_h)dp_h$, donc d'apr\`es \eqref{2.2.12}
\begin{equation}\label{2.2.12b}
\vert \mathbb E (f_{h})-\mathcal M_h\vert \leq r+C\exp(-{(N_{h}-1)r^2
\over \Vert f_h\Vert^2_{lips}})
\end{equation}
Il existe donc $a>0$ et $C_{1}>0$ tels que pour tout $r\in ]0,1]$ v\'erifiant
$r\geq C_{1}\exp(-aN_{h}r^2)$ on ait
\begin{equation}\label{2.2.12bis}
P_{h}(u\in S_{h}, \ \vert f_h(u)-\mathbb E(f_{h})\vert\geq 3r)\leq 2\exp(-aN_{h}r^2)
\end{equation}
En utilisant (\ref{2.2.1}), on en d\'eduit qu'il existe $M_{0}$ tel que pour tout
$M\geq M_{0}$, on a avec 
\be\label{rh}
r_{h}=MN_{h}^{-1/2}\vert \log(h)\vert^{1/2}+ C_{0}h^{-d+1}/N_{h}
\ee
\begin{equation}\label{2.2.12ter}
P_{h}(u\in S_{h}, \ \vert f_h(u)-m_{\sigma_{0}(A)})\vert\geq 4r_{h})\leq 2\exp(-aN_{h}r_{h}^2)
\end{equation}
et donc
\begin{equation}\label{2.2.12quad}
\mathbb P(\sup_{l\geq k} \vert f_{h_{l}}(u)-m_{\sigma_{0}(A)})\vert\leq 4\sup_{l\geq k}r_{h_{l}})\geq \Pi_{l\geq k}(1-2\exp(-aN_{h_{l}}r_{h_{l}}^2))
\end{equation}
On a d'apr\`es (\ref{rh}) $\lim_{l\rightarrow \infty} r_{h_{l}}=0$,
et aussi $\exp(-aN_{h_{l}}r_{h_{l}}^2)\leq h_{l}^{aM^2}$, donc
$$ \Pi_{l\geq k}(1-2\exp(-aN_{h_{l}}r_{h_{l}}^2) \geq 1-\varepsilon_k$$
avec $\lim_{k\rightarrow \infty}\varepsilon_k=0$.
 La preuve du th\'eor\`eme
\ref{thm2.1} est compl\`ete. 

\end{proof}

\subsection{Estimations $L^q$ presque sures}\label{sec2.3}

Dans cette section, on d\'emontre que les normes $L^q$ sont born\'ees presque surement si $q<+\infty$. On calcule aussi  l'ordre de grandeur des m\'edianes des fonctions 
$\Vert u\Vert_{L^q}$ sur $S_{h}$ pour $2\leq q \leq \infty$ .  On se limitera au cas $\mathbb{K}=\mathbb{C}$, les preuves dans le cas r\'eel \'etant similaire. 

Notons d'abord que les injections de Sobolev~\eqref{0.8} peuvent \^etre pr\'ecis\'ees en cas de localisation spectrale plus fine.
\begin{prop}\label{estimq}
Il existe $C_{0}>0$ tel que pour tout $q\in [2,\infty]$ et tout $u\in E_{h}$ on ait
\begin{equation}\label{eq.borne} 
\|  u \|_{L^q ( M)} \leq C_{0}N_{h}^{({1\over 2}-{1\over q})}\|u\|_{L^2(M)}. 
\end{equation}
\end{prop} 

En effet, pour $u\in E_{h}$, en utilisant (\ref{2.2}), l'in\'egalit\'e de Cauchy-Schwarz,
ainsi que (\ref{2.5bis}), on obtient 
\be
\vert u(x) \vert \leq e_{x,h}^{1/2}\Vert u\Vert_{L^2}\leq CN_{h}^{1/2}\Vert u\Vert_{L^2}
\ee
et la proposition r\'esulte de $\Vert u\Vert_{L^q}\leq \Vert u\Vert_{L^\infty}^{1-2/q}
\Vert u\Vert_{L^2}^{2/q}$
pour tout $q\in [2,\infty]$.

\begin{thm}\label{th.lq} Il existe $ c_1, c_2 >0$  tels que pour tout $q\in ]2,+\infty[$, 
si on note $\mathcal{M}_{q,h}$ la m\'ediane de la fonction $ F(u)= \|u\|_{L^q}$ sur la sph\`ere $\mathbb{S} _{h}$, on a 
 pour tout  $\Lambda \geq 0$, et tout $h \in ]0, 1]$
\begin{equation}\label{estlq} 
P_{h} ( u \in S_{h} ; \Bigl|\|u \|_{L^q} - \mathcal{M}_{q,h}\Bigr| >\Lambda ) \leq 2 e^{- c_1 N_h^{ \frac {2}   q} \Lambda^2}
\end{equation}
De plus, on a les estimations suivantes de la m\'ediane 
\be\label{est.med1}
1\leq \mathcal{M}_{q,h} \leq c_{2}\sqrt q \quad \forall q\in [2,\infty[, \ \forall h\in ]0,1]
\ee
Sous l'hypoth\`ese $\lim_{h\rightarrow 0}N_{h}=\infty$, il existe $c_{3}>0$, $q_{0}>2$
 et $\varepsilon_{0}>0$ tels que 
\be\label{est.med2}
 c_{3}\sqrt {q}\leq \mathcal{M}_{q,h} \quad \forall q\in [q_{0}, \varepsilon_{0}\log (N_{h})], \ \forall h\in ]0,1]
\ee
Enfin, sous l'hypoth\`ese $\lim_{h\rightarrow 0}{b_{h}-a_{h}\over h}=\infty$, pour tout $\gamma \in ]0,1[$, il existe $q_{\gamma}>2$ et $\varepsilon_{\gamma}>0$ tels que 
\be\label{est.med3}
\gamma \sqrt {q\over 2e \text{Vol}(M)}\leq \mathcal{M}_{q,h}\leq {1\over \gamma}
\sqrt {q\over 2e\text{Vol}(M)} \quad \forall q\in [q_{\gamma}, \varepsilon_{\gamma}\log (N_{h})], \ \forall h\in ]0,1]
\ee

\end{thm}
\begin{proof}

 D'apr\`es la proposition \ref{estimq}, on a
  \be\ba
 |F(u) - F(v) | &= |\|u\|_{L^q}- \|v\|_{L^q}| \\
 &\leq \|u-v\|_{L^q} \leq C_{0} N_h^{(\frac 1 2- \frac 1 q)} \|u-v\|_{L^2}
  \leq C N_h^{ (\frac 1 2- \frac 1 q)} \text{dist} (u,v)
   \ea\ee
   On en d\'eduit 
$$
  \|F\|_{\text{Lips}} \leq C N_h^{ (\frac 1 2- \frac 1 q)} 
$$
donc en utilisant le r\'esultat de
  concentration de la mesure (voir appendice A, proposition \ref{concentration})
 \`a la fonction  $ F(u)= \|u\|_{L^q}$ sur la sph\`ere $\mathbb{S} _{h}$ on obtient
\begin{equation}\label{eq.gdedev1}
 P_{h} ( |F(u) - \mathcal{M}_{q,h}| >r ) \leq 2 e^{- {N_{h}-1\over C^2N_{h}}N_h^{\frac 2 q} r^2}
 \end{equation}
 ce qui implique (\ref{estlq}). Pour l'estimation (\ref{est.med1}) de la m\'ediane, en utilisant
$$ \mathbb{E}_{h} (|g|^q) = q \int_0^{+\infty} \lambda^{q-1} P_{h} ( |g| >\lambda) d \lambda$$
on obtient d'apr\`es (\ref{2.3}) et (\ref{2.4}),
\begin{multline}\label{espmed1}
\mathbb{E}_{h} ( \|u\|_{L^q} ^q)
= \int_{S_{h} }\int_M \vert u(x)\vert ^q dx dp_{h} =  \int_M \int_{S_{h} } \vert u(x)\vert ^q dp_{h} dx\\
 =  q\int_M\int_{0}^\infty \lambda^{q-1}P_h ( | u(x)| >\lambda )d\lambda dx 
= q\int_M \int_{0}^{e_{x,h}^{1/2}}\lambda^{q-1} (1-{\lambda^2\over e_{x,h}})^{(N_{h}-1)}d\lambda dx \\
= q\Big( \int_M e_{x,h}^{q/2}\Big)\Big(\int_{0}^1 z^{q-1}
(1-z^2)^{(N_{h}-1)}dz\Big)=\mathcal A_{q,h}^q
\end{multline}
Le lecteur pourra v\'erifier qu'on a bien $\mathcal A_{2,h}=1$.  En notant
$B(x,y)=\int_{0}^1t^{x-1}(1-t)^{y-1}dt$ la fonction b\'eta, on a 
\be\label{espmed9}
\int_{0}^1 z^{q-1}
(1-z^2)^{(N_{h}-1)}dz= {1\over 2}B(q/2,N_{h})={\Gamma(q/2)\Gamma(N_{h})\over 2\Gamma(q/2+N_{h})}
\ee
donc en utilisant le lemme \ref{lem2.1} et la formule de Stirling,
\be\label{espmed2}
\mathcal A_{q,h}^q\leq q C^q\Bigl( {N_{h}\over {N_{h}+ \frac q 2}}\Bigr)^{N_{h}+q/2-1/2}\Gamma(q/2) \quad \text{d'o\`u} \quad
\mathcal A_{q,h} \leq C'\sqrt q.
\ee
D'apr\`es l'in\'egalit\'e de Bienaym\'e-Tchebychev, on a
 \begin{equation}\label{eq.tcheb}
 P_{h} ( \Vert u\Vert _{L^q} >t ) \leq \frac 1 {t^q} \mathbb{E} _{h} ( \Vert u\Vert^q_{L^q})
= \Bigl(\frac{ \mathcal A_{q,h}}{ t}\Bigr)^{q}.
\end{equation}
En choisissant $t = \mathcal{M}_{q,h}$, on obtient
 $$ \frac 1 2 \leq \lim_{\varepsilon \rightarrow 0}
 P_{h} ( \|u\|_{L^q}> \mathcal{M}_{q,h} -\varepsilon) \leq 
 \Bigl(\frac{ \mathcal A_{q,h}}{ \mathcal{M}_{q,h}}\Bigr)^{q}$$
 qui implique 
 \begin{equation}\label{medsup}
  \mathcal{M}_{q,h} \leq 2^{1/q}\mathcal A_{q,h}
  \end{equation}
donc $\mathcal{M}_{q,h}\leq c_{2}\sqrt q$ d'apr\`es (\ref{espmed2}). Comme $F(u)=\Vert u \Vert_{L^q}\geq \Vert u \Vert_{L^2}=1$, on a aussi $\mathcal{M}_{q,h}\geq 1$, ce qui prouve (\ref{est.med1}).

Prouvons \`a pr\'esent (\ref{est.med2}).  On remarque d'abord qu'on a en utilisant 
$\lim_{h\rightarrow 0}N_{h}=\infty$,  (\ref{espmed1}) et le lemme
\ref{lem2.1}, avec $0<a_{1}\leq a_{2}$ ind\'ependants de $h$ petit et de $q\in [2,\log (N_{h})]$

\be\label{aj1}
\mathcal A_{q,h} \in [a_{1}\sqrt q,a_{2}\sqrt q]
\ee
et
\begin{multline}
\bigl| \mathcal{A}_{q,h} - \mathcal{M}_{q,h} \bigr| ^q = \bigl| \|F\|_{L^q(\mathbb{S}_h)} - \|\mathcal{M}_{q,h}\|_{L^q(\mathbb{S}_h)} \bigr|^q\\
 \leq \| F - \mathcal{M}_{q,h} \|^q_{L^q(\mathbb{S}_h)} = q \int_0^{+\infty} \lambda^{q-1} P_h( |F- \mathcal{M}_{q,h}| >\lambda) d \lambda\\
\leq q \int_0^{+\infty} \lambda^{q-1} e^{-c_1 N_h^{\frac 2 q} \lambda^2} d\lambda
= {q\over 2N_{h}c_{1}^{q/2}}\Gamma( q/2).
\end{multline}
On en d\'eduit
\be\label{aj2}
\bigl| \mathcal{A}_{q,h} - \mathcal{M}_{q,h} \bigr| \leq \frac C { N_h^{1/ q}}  \sqrt{q}
\ee
qui implique (\ref{est.med2}) d'apr\`es~\eqref{aj1}. Montrons maintenant (\ref{est.med3}).
On a d'apr\`es~\eqref{2.5}
\be\label{aj5}
\Bigl(\int_{M}e_{x,h}^{q/2}dx \Bigr)^{1/q}= {N_{h}^{1/2}\over Vol(M)^{1/2-1/q}}(1+o_{h}(1)), \quad
\lim_{h\rightarrow 0} o_{h}(1)=0
\ee
et il en r\'esulte en utilisant (\ref{espmed1}), pour $q\leq \log (N_{h})$
\be\label{aj6}
\mathcal A_{q,h}= \sqrt {q\over 2e\text{Vol}(M)}((1+o_{h,q}(1)))
\ee
avec $\lim_{q\rightarrow\infty, h\rightarrow 0} o_{h,q}(1)=0$. L'encadrement
(\ref{est.med3}) r\'esulte alors  de (\ref{aj2}).
La preuve du th\'eor\`eme \ref{th.lq} est compl\`ete.

\end{proof}
\noindent
Nous allons \`a pr\'esent d\'eduire du th\'eor\`eme \ref{th.lq} une 
estimation $L^\infty$ qui compl\`ete le th\'eor\`eme \ref{thm0.2}.

\begin{thm}\label{th.linfini}
On suppose qu'il existe $a>0$ tel que $N_{h}\geq h^{-a}$ pour $h$ petit (cette hypoth\`ese est toujours v\'erifi\'ee en dimension $d\geq 2$).
Il existe $h_{0},c_{0},c_{1}>0$ tels que pour tout $h\in ]0,h_{0}]$ on ait
\be\label{aj7}
\mathbb E_{h}(\Vert u \Vert_{L^\infty}) \in [c_{0} \sqrt{\vert\log(h)\vert},c_{1}
\sqrt{\vert\log(h)\vert}] 
\ee 
et en notant $\mathcal M_{\infty,h}$ la m\'ediane de $F_{\infty}(u)=\Vert u \Vert_{L^\infty}$
\be\label{aj8}
\mathcal M_{\infty,h} \in [c_{0}\sqrt{\vert\log(h)\vert},
c_{1}\sqrt{\vert\log(h)\vert}] 
\ee 
\end{thm}
\begin{proof}
Il existe $C_{0}>0$ tel que pour tout $h\in ]0,1]$,  tout $q>2$,
et tout $u\in E_{h}$ on a 
$$\Vert u \Vert_{q}\leq \Vert u \Vert_{L^\infty}\leq C_{0}{h}^{-d/q}\Vert u \Vert_{q}$$
donc en choisissant $q_{h}=a\varepsilon_{0}\vert\log (h)\vert$, 
on obtient avec $C_{1}=C_{0}e^{d/{a\varepsilon_{0}}}$ ind\'ependant de $h$
\be\label{aj9}
\mathbb E_{h}(\Vert u \Vert_{q_{h}})\leq \mathbb E_{h}(\Vert u \Vert_{L^\infty})
\leq C_{1}\mathbb E_{h}(\Vert u \Vert_{q_{h}})
\ee
Or pour tout $q\in [2,\infty]$, on a par le r\'esultat de concentration de la mesure
d'apr\`es (\ref{eq.gdedev1}) (qui s'applique aussi \`a $q=\infty$), avec $c_{1}$
ind\'ependant de $h$
\be\label{aj10}
\ba
\vert \mathbb E_{h}(\Vert u \Vert_{q})-\mathcal M_{q,h}\vert
&\leq \int \vert \Vert u \Vert_{q}-\mathcal M_{q,h}\vert dP =
\int _{0}^\infty  P_{h} (\vert \Vert u \Vert_{q}-\mathcal M_{q,h}\vert>\lambda)d\lambda\\
&\leq 2\int _{0}^\infty e^{-c_{1}\lambda^2}=\sqrt {\pi/c_{1}}
\ea\ee
Comme on a $q_{h}\leq \varepsilon_{0}\log(N_{h})$,  (\ref{est.med1}) et (\ref{est.med2})
impliquent $\mathcal M_{q_{h},h}\simeq 
\sqrt{q_{h}}$. On d\'eduit alors de (\ref{aj10}) pour $h$ petit $\mathbb E_{h}(\Vert u \Vert_{q_{h}})\simeq \sqrt{q_{h}}$, donc aussi d'apr\`es (\ref{aj9}), $\mathbb E_{h}(\Vert u \Vert_{L^\infty})\simeq \sqrt{q_{h}}$, donc \`a nouveau par (\ref{aj10}) $\mathcal M_{\infty,h}\simeq \sqrt{q_{h}}$. 
La preuve du th\'eor\`eme \ref{th.linfini} est compl\`ete.

 \end{proof}
\section{Application aux harmoniques sph\'eriques}\label{sec3}
Rappelons que les fonctions propres du laplacien sur la sph\`ere $\mathbb{S}^d$,  (muni de sa m\'etrique standard) sont les harmoniques sph\'eriques de degr\'e $k$ (restrictions \`a la sph\`ere $\mathbb{S}^{d}$ des polynomes harmoniques de degr\'e $k$) et qu'elles v\'erifient:
\begin{itemize}
\item Pour tout $k\in \mathbb{N}$, l'espace vectoriel sur $\mathbb{K}$ des harmoniques sph\'eriques de degr\'e $k$, $E^k$, est de dimension 
\begin{equation}\label{dimension}N_k= \binom{k+d}{d} - \binom{k+d-2}{d}\sim_{k\rightarrow + \infty} \frac {2} {(d-1)!} k^{d-1}
\end{equation}
\item Pour tout $e\in E^k$, $-\mathbf{\Delta}_{\mathbb{S}^d} e=  k (k+d-1)e$.
\item L'espace vectoriel engendr\'e par les harmoniques sph\'eriques est dense dans $L^2( \mathbb{S}^d)$
\end{itemize} 
On peut identifier l'espace des bases orthonormales de l'espace $E^k$ (muni de la norme $L^2$) avec le  groupe  unitaire, $U(N_k) $, si $\mathbb{K} = \mathbb{C}$ et avec le groupe  orthogonal, $O(N_k)$, si $\mathbb{K}= \mathbb{R}$. On munit $U(N_k)$ (resp. $O(N_k)$) de sa mesure de Haar, $\nu_k$ (resp. $\widetilde{\nu} _k$). On peut alors identifier l'espace des bases Hilbertiennes  de $L^2(M; \mathbb{K})$ form\'ees d'harmoniques sph\'eriques avec 
$$\mathcal{B}= \times_{k\in \mathbb{N}}U(N_k) , \qquad (\text{resp. } \widetilde{\mathcal{B}}=\times_{k\in \mathbb{N}}O(N_k))
$$
qu'on munit de la mesure de probabilit\'e produit
$$\nu= \otimes_k \nu_k, (\text{resp.}  \widetilde{\nu}= \otimes_k \widetilde{\nu}_k)
$$ 
On a alors le r\'esultat suivant
\begin{thm} \label{thm4} Soit $d\geq 2$. Il existe $c>0$ et pour tout $2\leq q <+ \infty$,
$C_{q}>0$ tels que
$$\nu(\{ B= (b_{k,l})_{k\in \mathbb{N}, l=1, \cdots N_k}\in \mathcal{B}; \exists k,l; \Bigl| \|b_{k,l} \|_{L^q( \mathbb{S}^d)} - \mathcal M_{q,k} \Bigr| > r\} ) \leq C_{q} e^{-cr^2}
$$
$$\widetilde{\nu}(\{ B= (b_{k,l})_{k\in \mathbb{N}, l=1, \cdots N_k}\in \widetilde{\mathcal{B}}; \exists k,l; \Bigl| \|b_{k,l} \|_{L^q( \mathbb{S}^d)} - \mathcal M_{q,k} \Bigr| > r\} ) 
\leq C_{q} e^{-cr^2}
$$
o\`u la m\'ediane (dans le cas complexe) $\mathcal{M}_{q,k}$ v\'erifie avec $C$ ind\'ependant de $q,k$
\be\label{est.med4}\ba
&\bigl| \mathcal{A}_{q,k} - \mathcal{M}_{q,k}\bigr|\leq  \frac{C} { N_{k}^{1 /q}}\sqrt{q}\\
&\mathcal{A}_{q,k}= \frac{ N_{k}^{\frac 1 2 }}{(\text{Vol}(\mathbb{S}^{d}))^{\frac  1 2 -\frac 1 q} }\Bigl(q\frac{\Gamma(q/2)\Gamma(N_{k})} {2\Gamma(q/2+N_{k})}\Bigr)^{\frac 1 q}
  \ea\ee
\end{thm}

\begin{proof}
 On se limite au cas complexe, $\mathbb{K} = \mathbb{C}$, le cas r\'eel \'etant similaire. 
 On commence par remarquer que, quitte \`a augmenter la constante $C_{q}$, on peut se contenter de prouver le th\'eor\`eme pour les 
 bases orthonormales du sous espace $\oplus_{k\geq 1}E^k$ de $L^2$. On prend $h_k= k^{-1}$. On remarque alors que 
$$\sqrt{ k (k+d-1)} = h_{k}^{-1} \sqrt{ 1+ (d-1) h_k} = h_k^{-1} + \frac{ d-1} 2 + O(h_k). $$
On en d\'eduit qu'avec 
$$a_h = 1+ h \bigl(\frac{d-1} 2 - \frac 1 4\bigr), \qquad b_h =  1+ h \bigl(\frac{d-1} 2 + \frac 1 4\bigr) 
$$
on a pour $h_k = k^{-1}, k \in \mathbb{N^*}$, $E_{h_k}=E^k$ pour $k$ assez grand. 
 On remarquera que l'hypoth\`ese ~\eqref{eq.mino} avec $D$ assez grand n'est plus v\'erifi\'ee,  mais cela est sans importance car on a une meilleure estimation pour $e_{x,h}$ que celle donn\'ee par le lemme~\ref{lem2.1}:
 \begin{lem}\label{lem6.7}
 Pour tout $x\in \mathbb{S}^d$ et tout $k\in \mathbb{N}$, on a 
 $$e_{x,h_k} = \frac{N_{k}}{\text{Vol($M$)}}.$$
 \end{lem}
 En effet, le noyau du projecteur orthogonal sur l'espace $E^k$ est donn\'e par 
$$ K(x,y)= \sum_{j=1}^{N_{h_k}} e_{j,k}(x) \overline{e_{j,k}(y)},$$
o\`u $(e_{j,k})_{j=1}^{N_{h_k}}$ sont une base orthornomale de $E^k$. 
L'espace $E^k$ \'etant invariant par l'action des rotations de la sph\`ere, on en d\'eduit que $K(x,y)= K(Rx, Ry)$ pour toute rotation $R$. La fonction $x\in \mathbb{S}^d \mapsto e_{x,h_k}$ est donc constante, d'int\'egrale $N_{k}$, d'o\`u le lemme~\ref{lem6.7}. On en d\'eduit que les th\'eor\`emes \ref{thm0.2} et \ref{th.lq} s'appliquent, avec d'apr\`es~\eqref{espmed1},~\eqref{espmed9} et~\eqref{aj2}
$$\mathcal{A}_{q,k}= \frac{ N_{k}^{\frac 1 2 }}{(\text{Vol}(\mathbb{S}^{d}))^{\frac  1 2 -\frac 1 q} }\Bigl(q\frac{\Gamma(q/2)\Gamma(N_{k})} {2\Gamma(q/2+N_{k})}\Bigr)^{\frac 1 q}, \quad 
 \bigl| \mathcal{A}_{q,k} - \mathcal{M}_{q,k}\bigr|\leq  \frac{C} { N_{k}^{1 /q}}\sqrt{q}$$
En particulier, le  th\'eor\`eme~\ref{th.lq} implique 
\begin{prop} Il existe $c_{0},c_{1}>0$ tels que pour tout $k\geq 1$ et tout $q\geq 2$, on a 
\begin{equation}\label{eq.bon} \nu_k ( \{ B= (b_{l})_{l=1}^{N_k} \in U(N_k); \exists 1\leq  l\leq N_k;  \Bigl|\|b_l \|_{L^q} - \mathcal M_{q,k}\Bigr| >\Lambda ) \leq  c_0 e^{- c_1 k^{ \frac {2(d-1)}   q} \Lambda^2}k^{d-1}
\end{equation}
\end{prop}
En effet, l'application 
$$ B= (b_l)_{l=1}^{N_k} \in U(N_k) \mapsto b_{1} \in S_k$$
envoie la mesure $\nu_k$ sur la mesure $P_{h_k}$ et donc d'apr\`es le th\'eor\`eme~\ref{th.lq} pour tout $1\leq l_0\leq N_k$
$$ \nu_k ( \{ B= (b_l)_{l=1}^{N_k}\in U(N_k);   \Bigl|\|b_{l_0} \|_{L^q} - \mathcal M_{q,k}\Bigr| >\Lambda ) \leq  2 e^{- c_1 k^{ \frac {2(d-1)}   q} \Lambda^2}.
$$
On en d\'eduit que 
$$ \nu_k ( \{ B= (b_l)_{l=1}^{N_k}\in U(N_k); \exists l_0\in \{1, \dots, N_k\};  \Bigl|\|b_{l_0} \|_{L^q} - \mathcal M_{q,k}\Bigr| >\Lambda ) \leq  2 e^{- c_1 k^{ \frac {2(d-1)}   q} \Lambda^2}N_k,
$$
ce qui implique clairement~\eqref{eq.bon}. 

On d\'efinit maintenant les ensembles 
$$F_{k,r}= \{ B= (b_k)\in U(N_{h});  \forall  l_0; \Bigl|\|b_{k, l_0} \|_{L^q} - \mathcal M_{q,k} \Bigr| \leq r\}
$$
et 
$$ F_r= \cap_{k\geq 1} F_{k,r}= \{ B= (b_{k,l}); k \in \mathbb{N}^*, l =1 , \cdots N_k; \forall  k_0, l_0; \Bigl|\|b_{k_0, l_0} \|_{L^q} - \mathcal M_{q,k}\Bigr| \leq r\}
$$
On a pour $r\geq 1$
\begin{equation}
\nu(F_r^c) \leq \sum_{k\geq 1} \nu_k ( F_{k,r}^c) \leq \sum_{k\geq 1}  c_0 e^{- c_1 k^{ \frac {2(d-1)}   q} r^2}k^{d-1}
\leq \sum_{k\geq 1} C_{q} e^{-c_{1} k^{\frac {2(d-1)} q} r^2/2} \leq C_{q} e^{- c_{1} r^2/2}
\end{equation}
Quitte \`a augmenter $C_{q}$ on a aussi  $1\leq C_{q}e^{-c_{1}r^2/2}$ pour tout $r\leq 1$.
La preuve du th\'eor\`eme~\ref{thm4} est compl\`ete. 
\end{proof}

On d\'emontre de la m\^eme fa\c con (en rempla\c cant le Th\'eor\`eme~\ref{th.lq} par le Th\'eor\`eme~\ref{thm0.2} le r\'esultat suivant (on renvoie \`a~\cite{VdK1} pour un r\'esultat similaire sur $\mathbb{S}^2$, avec une estim\'ee en $\log^2(k)$ au lieu de $\log^{1/2} (k)$).
\begin{thm}\label{thm7}$\exists C,c,c_{0}>0, \forall r\geq 0$
$$\nu(\{ B= (b_{k,l})_{k\in \mathbb{N}, l=1, \cdots N_k}\in \mathcal{B}; \exists k,l;  \|b_{k,l} \|_{L^\infty( \mathbb{S}^d)}   > c_0 \log^{1/2}( k)+ r\} ) \leq C e^{-cr^2}
$$
$$\widetilde{\nu}(\{ B= (b_{k,l})_{k\in \mathbb{N}, l=1, \cdots N_k}\in \widetilde{\mathcal{B}}; \exists k,l;  \|b_{k,l} \|_{L^\infty( \mathbb{S}^d)}   > c_0 \log^{1/2}( k)+ r\} ) \leq C e^{-cr^2}
$$
\end{thm}

On d\'eduit aussi du th\'eor\`eme~\ref{thm4} le r\'esultat suivant.

\begin{thm}\label{thm.extra} La probabilit\'e qu'on puisse extraire d'une base orthonormale une sous suite born\'ee en norme $L^\infty$ est nulle: pour toute constante $C>0$,
 $$\nu(\{ B= (b_{k,l})_{k\in \mathbb{N}, l=1, \cdots N_k}\in \mathcal{B}; \liminf_{k\rightarrow+\infty} \inf_{l=1, \cdots, N_k}  \|b_{k,l} \|_{L^\infty( \mathbb{S}^d)} \leq C\}) =0
$$
$$\widetilde{\nu}(\{ B= (b_{k,l})_{k\in \mathbb{N}, l=1, \cdots N_k}\in \mathcal{B}; \liminf_{k\rightarrow+\infty} \inf_{l=1, \cdots, N_k} \|b_{k,l} \|_{L^\infty( \mathbb{S}^d)} \leq C\}) =0
$$
\end{thm}
\begin{proof} On se limite au cas complexe. Notons $A_{k,C}$ l'\'ev\'enement de $U(N_{k})$ 
$$A_{k,C} =\{B= (b_{k,l}),\exists l\in \{1, \cdots, N_k\}, \quad \|b_{k,l} \|_{L^\infty( \mathbb{S}^d)} \leq C\}$$
D'apr\`es le th\'eor\`eme \ref{th.lq}, il existe $k_{C}$ et $q=q_{C}$ tels qu'on ait \`a la fois
$\text{Vol}(\mathbb{S}^d)^{1/q}\leq 2$ et $\mathcal M_{q,k}\geq 3C$ pour tout $k\geq k_{C}$.
Comme on a $\|u\|_{L^q(\mathbb{S}^d)} \leq \text{Vol}(\mathbb{S}^d)^{\frac 1 q} \|u\|_{L^\infty}$, il r\'esulte de (\ref{estlq}) qu'on a 
\be
\nu_{k}(A_{k,C})\leq 2N_{k}e^{-c_{1}N_{k}^{2/q}C^2} \quad \forall k\geq k_{C}
\ee
La s\'erie $\sum_{k }N_{k}e^{-c_{1}N_{k}^{2/q}C^2}$ \'etant convergente, il en r\'esulte
$\lim_{L\rightarrow \infty}\nu (\cup_{k\geq L}A_{k,C})=0$, donc
$\nu(\cap_{L}\cup_{k\geq L}A_{k,C})=0$. La preuve du th\'eor\`eme \ref{thm.extra}
est compl\`ete.

\end{proof}
\begin{remk}\label{rem.4}
Les r\'esultats de cette section restent vrais sur une vari\'et\'e riemannienne compacte quelconque, sous une forme plus faible. En effet, si on d\'ecompose $L^2(M)$ en somme directe:
$$ L^2(M) = \oplus_k E_{h_k}$$ o\`u les espaces $E_{h_k}$ correspondent \`a un choix de $b_h-a_h= M h$, la m\^eme preuve donnera que presque toute base hilbertienne de $L^2(M)$ respectant cette d\'ecomposition aura toutes ses normes $L^q$ uniform\'ement born\'ees. La principale diff\'erence avec le cas des sph\`eres \'etant que les \'el\'ements de $E_{h_k}$ ne sont plus des fonctions propres exactes, mais des fonctions propres approch\'ees: 
$$u \in E_{h_k} \Rightarrow \sqrt{ - \mathbf{\Delta}} u = h_k^{-1} u + O(1)_{L^2(M)}$$
Sur les vari\'et\'es de Zoll (o\`u toutes les g\'eod\'esiques sont p\'eriodiques), la r\'epartition en paquets des fonctions propres (voir~\cite{Be}) permet d'am\'eliorer cette propri\'et\'e 
$$u \in E_{h_k} \Rightarrow  - \mathbf{\Delta} u = \omega_k^2 u+O(1)_{L^2(M)}$$
\end{remk}
\subsection{Estim\'ees sur les tores $\mathbb{T}^d$}\label{sec.3.2}
Notons pour $k\in \mathbb{N}$, $E_k$ l'espace propre du laplacien sur le tore $\mathbb{T}^d$ associ\'e \`a la valeur propre $k$,  qu'on munit de la norrme $L^2(\mathbb{S}^d)$ et $N_k$ sa dimension,
$$N_k = \sharp J_k, \qquad J_k = \{n\in \mathbb{Z}^d; |n|^2 = k\}.$$
qui est le nombre de pr\'esentation de l'entier $k$ en somme de $d$ carr\'es d'entiers. On note 
$$\mathfrak{S}_d= \{ k\in \mathbb{N}; N_k\geq 1\},$$
le spectre du Laplacien et pour $k\in \mathfrak{S}_d$, on munit alors les sph\`eres unit\'es $\mathbb{S}_k$ des espaces $E_k$ de la probability uniforme, l'espace des suites d'\'el\'ements de $\mathbb{S}_k, k\in \mathfrak{S}_d$ de la probabilit\'e produit, et l'espace des bases orthonormales form\'ees de fonctions propres de la mesure de probabilit\'e naturelle comme pr\'ec\'edemment.
\subsubsection{Preuve de~\eqref{eq.mino-tores}}
 
On a clairement d'apr\`es la formule de Weyl
$$ \sum_{\frac n 2 < m\leq n } N_m =  C n^{\frac d 2} + O(n^{\frac{ d-1} 2})
$$ 
On en d\'eduit que pour tout $d\geq 2$ il existe une suite $n_p \rightarrow + \infty$ telle que 
\be\label{ttt}
N_{n_p} \geq c n_p ^{\frac{d-2} 2}
\ee
et en consid\'erant la fonction propre du laplacien sur $\mathbb{T}^k$ (de norme $L^2$ \'egale \`a $1$),
$$ u_p(x) =N_{n_p}^{-\frac 1 2}  \sum_{N= (n_1, \dots, n_k)\in E_{n_p} }e^{i N \cdot x} 
$$ on voit d'apr\`es l'injection de Sobolev que 
$$  c n_p^{\frac{d-2} 4} \sim \|u_p\|_{L^\infty} \leq C n_p^{\frac k {2r}} \|u_p\|_{L^r}
$$ 
Donc
$$
c n_p^{\frac{d-2} 4 - \frac k {2r}}  \leq  C\|u_p\|_{L^r}
$$
ce qui est~\eqref{eq.mino-tores}.
\subsubsection{Bases orthonormales sur les tores $\mathbb{T}^d$}
Il est clair si on choisit la base orthonormale de $E_k$ donn\'ee par 
$$B= \{ e_n (x)= \frac{1} {\sqrt{\text{Vol}(\mathbb{T}^d)}}e^{in\cdot x}, n\in J_k\},$$
qu'on a 
$$e_k(x) = \frac{1} {\text{Vol}(\mathbb{T}^d)} \sum_{n\in J_k} | e^{in\cdot x}|^2 = \frac{1} {\text{Vol}(\mathbb{T}^d)}\sum_{n\in{J_k}} 1= \frac{N_k} {\text{Vol}(\mathbb{T}^d)}.$$
On en d\'eduit que tous les r\'esultats des section~\ref{sec2.1} et~\ref{sec2.3} s'appliquent dans ce cadre. La seule diff\'erence par rapport au cas des sph\`eres est que  $\lim_{k\rightarrow+ \infty} N_k =+\infty$, n'est vrai que pour $d\geq 5$ (et alors $N_k \geq ck^{\frac d 2 -1}$ voir Grosswald~\cite[Chapter 12]{Gr}). La situation est donc un peu moins favorable et l'analogue des Th\'eor\`emes~\ref{thm4} et~\ref{thm7} n'est donc vrai que si $d\geq5$. En revanche, le r\'esultat plus faible suivant est vrai en toute dimension $d\geq 2$ (et la d\'emonstration est cons\'equence imm\'ediate des r\'esultats des sections~\ref{sec2.1}, ~\ref{sec2.3}) et de $\limsup_{k\rightarrow \infty}N_{k}=\infty$ d\`es que $d\geq 2$.
\begin{thm}
Consid\'erons l'espace des suites $U= (u_k)_{k\in \mathfrak{S}_d}; \forall k\in \mathfrak{S}_d,  u_k \in \mathbb{S}_k$ muni de sa mesure de probabilit\'e naturelle $\nu$. Alors 
$$\nu(\{ U; \exists C>0; \forall k, \|u_k\|_{L^\infty ( \mathbb{T}^d)} \leq C \}) =0
$$
\end{thm} 
\section{Application \`a la stabilisation des ondes}\label{sec4}

On s'interesse dans cette section \`a l'\'etude de la d\'ecroissance de
l'\'energie pour les solutions de l'\'equation des ondes amorties sur la vari\'et\'e riemannienne compacte connexe $(M,g)$:
\begin{equation}\label{5.1}
(\partial_t^2 - \mathbf{\Delta} ) u + 2a(x) \partial_t u=0 , \quad u \mid_{t=0}= u_0\in H^1(M), \partial_t u \mid_{t=0} = u_1 \in L^2 (M)
\end{equation}
o\`u $a\in C^\infty (M; [0,\infty[) $ est non identiquement nulle. Pour 
$u=(u_0,u_1)\in H^1(M)\times L^2 (M)$, l'\'energie de $u$  est 
\begin{equation}\label{5.2}
\mathcal{E}(u)={1\over 2 }\int_M( |\nabla_x u|^2 + |\partial _t u|^2) d_gx
\end{equation}
On a $\mathcal{E}((u_0,u_1))=0$ si et seulement si $u_1=0$ et $u_0(x)$ est une constante, 
et pour $u(t,x)$ solution de~\eqref{5.1},

\begin{equation}\label{5.3}
\mathcal{E}(u(0,.)-\mathcal{E}(u(t,.))=\int_0^t \int _M 2a(x) |\partial _s u(s,x) |^2 d_gx
\end{equation}

On renvoie \`a \cite{leb1} et \cite{Sj1} pour des r\'esultats g\'en\'eraux sur
la d\'ecroissance de l'\'energie et la th\'eorie spectrale de 
l'\'equation des ondes amorties, et \`a \cite{AL} pour des r\'esultats
num\'eriques. Rappelons simplement ici que d\`es que 
$a\not \equiv 0$, on  a pour toute solution $u$ de \eqref{5.1}
\begin{equation}\label{5.0}
\lim_{t \rightarrow + \infty} \mathcal{E}(u(t,.)) =0
\end{equation}
et que si le support de $a$ contr\^ole g\'eom\'etriquement la vari\'et\'e $M$, c'est \`a dire si toute g\'eod\'esique rencontre l'ensemble $\{x\in M; a(x) >0\}$, alors la d\'ecroissance de l'\'energie est uniforme (donc exponentielle):
\begin{equation}\label{5.00} \exists C, \epsilon >0, \ \ \forall u=(u_0, u_1) \in H^1(M) \times L^2(M), \qquad \mathcal{E}(u(t,.)) \leq C e^{- \epsilon t}\mathcal{E}(u(0,.))
\end{equation} 

Nous allons \'etudier ici la d\'ecroissance de l'\'energie d'un point de vue probabiliste, et uniform\'ement par rapport \`a la fr\'equence. Soit
$A_a$ l'op\'erateur non born\'e sur l'espace de Hilbert $H^1(M)\times L^2(M)$
\begin{equation}\label{5.4}
A_a=\begin{pmatrix}0 
&Id \\ 
\mathbf{\Delta} 
& -2a 
\end{pmatrix}, \ D(A_a)=H^2\oplus H^1
\end{equation}
On notera $U(t)=e^{tA_a}$, de sorte que la solution du syst\`eme~\eqref{5.1} avec donn\'ees
$(u_0,u_1)$ v\'erifie $(u(t,.),\partial_t u(t,.))=U(t)(u_0,u_1)$.
On a pour tout $t\in\mathbb R$, $U(t)(1,0)=(1,0)$, de sorte que $U(t)$ op\`ere
sur l'espace de Hilbert $H=H^1(M)/{\mathbb C 1}\oplus L^2(M)$, et on 
consid\'erera aussi $A_a$ donn\'e par \eqref{5.4} comme op\'erateur non born\'e
sur $H$. Ce passage au quotient est justifi\'e par le fait que bien que
la quantit\'e $\int_M u(t,.)d_gx$ ne soit pas
pr\'eserv\'ee pour les solutions de \eqref{5.1}, la moyenne de $u$
sur $M$ n'intervient pas dans le calcul de l'\'energie. Le produit scalaire sur $H$ est d\'efini par $(u\vert v)_\mathcal{E}={1\over 2}\int_M (\nabla_x u_0 \overline \nabla_x v_0 +u_1\overline v_1)d_gx$ avec $u=(u_0,u_1), v=(v_0,v_1)$ o\`u $u_0,v_0$ sont d\'efinis modulo des constantes, et
on a $\mathcal{E}(u)=(u\vert u)_\mathcal{E}$. On notera $\Vert u\Vert_\mathcal{E}=\sqrt{\mathcal{E}(u)}$.
Pour $u=(u_0,u_1)\in H$, on identifiera $u_0\in H^1(M)/{\mathbb C 1}$ avec
la fonction sur $M$, $\sum_{\omega_n>0} (\int_M u_{0}e_{n}(x)d_gx) e_n$, qui est
ind\'ependante du choix du repr\'esentant de $u_0\in H^1(M)$.

On choisit ici $0<a_h= c<c' = b_h$ et on rappelle que  
 $\widetilde E_{h}$ est le $\mathbb{R}$ sous espace vectoriel
de $H$ engendr\'e par les fonctions propres $e_n$ du laplacien sur $M$
avec $h\omega_n \in [c,c']$. Pour
$u_0=\sum_{\omega_n>0} u_{0,n}e_n, u_1=\sum_{\omega_n\geq 0} u_{1,n}e_n$, on notera aussi $\widetilde 
\Pi_h$ 
l'op\'erateur
born\'e de $H$ sur $\widetilde E_h$ 
\begin{equation}\label{5.5}
\widetilde\Pi_h(u)=(\sum_{h\omega_n \in I} u_{0,n}e_n, \sum_{h\omega_n \in I}
 u_{1,n}e_n)
\end{equation} 
et $i_h$ l'injection de $\widetilde E_h$ dans $H$. Alors 
$\widetilde\Pi_h$ est le projecteur orthogonal sur $\widetilde E_h$
 et $i_h$ est son adjoint. On prendra garde au fait que le 
sous-espace 
$\widetilde E_h$ de $H$ n'est pas invariant par le groupe d'\'evolution $U(t)$. Toutefois,
la notion d'\'echelle est conserv\'ee uniform\'ement en $t\geq 0$ par
le groupe $U(t)$. Plus pr\'ecis\'ement, on montrera \`a la fin de ce
paragraphe le th\'eor\`eme suivant.
\begin{thm}\label{thm5.0}
Soient $I=[c,c'], 0<c<c'$, et $r>0$ fix\'es.  Pour tout entier $N\geq 1$, il existe une constante $C_N$,
telle que pour tout $0<h<h'<1$ avec $ch'-c'h \geq rh$, 
 on a 
\begin{equation}\label{5.6}
\sup_{t\geq 0}\Vert \widetilde\Pi_h U(t)i_{h'}\Vert_\mathcal{E}
+\sup_{t\geq 0}\Vert \widetilde\Pi_{h'} U(t)i_h\Vert_\mathcal{E}
\leq C_N h^N 
\end{equation}
\end{thm}
\subsection{Amortissement presque sur}
Soit $S^*(M)=\{(x,\xi)\in T^*M, \ \vert \xi_x\vert=1\}$ le fibr\'e
cotangent unitaire et $\phi(t,\rho)$ le flot g\'eod\'esique sur $S^*(M)$. 
Soit $\underline a(t,\rho)$ la moyenne de la fonction $a$ \`a
l'instant $t$ le long du flot
\begin{equation}\label{5.7}
\underline a(t,\rho)={1\over t}\int_0^ta(x(\phi(s,\rho)))ds
\end{equation} 
La mesure de Liouville  canonique $d\mu$ sur $S^*(M)$ \'etant invariante
par le flot g\'eod\'esique,  la fonction
\begin{equation}\label{5.8}
Bir(\rho)=\lim_{t\rightarrow \infty}\underline a(t,\rho)
\end{equation}
existe $\mu$-presque partout d'apr\`es le th\'eor\`eme ergodique de 
Birkhoff.

On notera $\widetilde S_h$ la sphere unit\'e de $\widetilde E_h$, et
$\widetilde P_h$ la probabilit\'e uniforme sur $\widetilde S_h$. Comme application simple
des r\'esultats de la section \ref{sec2}, on a le th\'eor\`eme
suivant qui relie le comportement de la fonction de Birkhoff \`a
la d\'ecroissance de l'\'energie.

\begin{thm}\label{thm5.1}
On suppose  que la fonction de Birkhoff v\'erifie
\begin{equation}\label{5.9}
\mu(\rho\in S^*(M), \ Bir(\rho)=0)=0
\end{equation}
Alors pour tout $\varepsilon>0, \alpha\in ]0,1]$, il existe un temps $T>0$,
tel que
\begin{equation}\label{5.10}
\forall {h\in ]0,1]}, \widetilde P_h\Big(u\in \widetilde S_h, \ \mathcal{E}(U(T)u)<\varepsilon  \Big)\geq 1-\alpha
\end{equation}
\end{thm}

\begin{remk}On remarquera que le temps $T$ \`a partir duquel on demande 
que l'\'energie soit plus petite que $\varepsilon$ est ind\'ependant de $h\in ]0,1]$, de sorte que l'estimation \eqref{5.10} est diff\'erente de la simple estimation \eqref{5.0}.  Dans le cas o\`u
$M$ est une  sph\`ere, et la fonction $a$ identiquement nulle dans un voisinage de l'\'equateur, l'hypoth\`ese \eqref{5.9} est viol\'ee.  On laisse au lecteur le soin de prouver dans ce cas que l'estimation \eqref{5.10}
est fausse pour $\alpha$ et $\varepsilon $ petits. 
\end{remk}

\begin{thm} \label{thm5.2}
Dans le cas o\`u  le flot est ergodique, ce qui implique que $\mu$ presque surement 
$$\text{Bir} (\rho) = \frac{ \int_M a(x) dx } { \text{Vol} (M)}=\overline a>0$$ et donc \eqref{5.9} est v\'erifi\'ee, on suppose de plus  l'estim\'ee quantitative suivante sur le caract\`ere exponentiellement m\'elangeant du flot: 
\begin{equation}\label{eq.disp}
\forall \delta >0, \exists S( \delta)>0; \mu ( \rho\in S^*M; |\underline a (t, \rho ) - \overline a| > \delta) \leq e^{-t S( \delta)}
\end{equation}
Alors il existe  $\mathcal{M}, \alpha, C, c >0,$ tel que pour tout $h \in ]0,1]$, 
\begin{equation}
  P_h(\sup_{t\in [0, \epsilon \log(1/h)]}(\mathcal{E}(U(t)u) e^{t\alpha} )>\mathcal{M}+r)\leq C e^{-ch^{-d+1}r^2} 
\end{equation}
\end{thm}
\begin{remk} L'hypoth\`ese est v\'erifi\'ee dans le cas o\`u la vari\'et\'e $M$ est une surface de courbure n\'egative constante et l'amortissement n'est pas trivial (voir Kifer~\cite{Ki} et Anantharaman~\cite[Theorem 3.1]{An}.
\end{remk}  
\begin{proof} On commence par d\'emontrer le Th\'eor\`eme~\ref{thm5.1}. Pour cela on va v\'erifier
le lemme suivant.
\begin{lem}\label{lem5.1}
On suppose que l'hypoth\`ese \eqref{5.9} est satisfaite. Il existe $C,c>0$
et pour tout $\varepsilon>0$, $T_\varepsilon$ et
$h_\varepsilon$ tels 
qu'on ait pour tout $h\in ] 0 ,h_\varepsilon]$
\begin{equation}\label{5.11}
\widetilde P_h\Big(u\in \widetilde S_h, \ \mathcal{E}(U(T_\varepsilon)u)<\varepsilon  \Big)\geq 1-Ce^{-ch^{-d}\varepsilon^2}
\end{equation}
\end{lem}
\begin{proof}
Soit $\lambda=\sqrt {-\mathbf{\Delta}}$ et $j$ l'application de 
$H^1(M)\oplus L^2(M)$ dans $L^2(M)\oplus L^2(M)$ d\'efinie par
\begin{equation}\label{5.12}
j(u)=v \Longleftrightarrow v_0={\lambda u_0-iu_1\over \sqrt 2}, \ 
v_1={\lambda u_0+iu_1\over \sqrt 2}
\end{equation}
On remarquera que $j$ est d\'efinie sur $H=H^1(M)/{\mathbb C 1}\oplus L^2(M)$
et induit une isom\'etrie de
$\widetilde E_h$ sur $E_h\oplus E_h$.
Posons $d_\pm=\pm i\lambda-a$,  
$\mathcal D=
\begin{pmatrix}
d_+ &0 \\ 
0 & d_-
\end{pmatrix}$ et $\mathcal R=
\begin{pmatrix}
0 &a \\ 
a& 0
\end{pmatrix}$. Pour $u\in \widetilde E_h$, soit $v=j(u)$ et 
$v(t)=j(U(t)u)$. On v\'erifie ais\'ement que $v(t)$ v\'erifie l'\'equation 
\begin{equation}\label{5.13}
\partial_t v(t)=(\mathcal D + \mathcal R)v(t)
\end{equation}
Les op\'erateurs $d_\pm$ sont des op\'erateurs pseudo-diff\'erentiels classiques elliptiques de degr\'e $1$, on $d_+^*=d_-$ (adjoint sur $L^2$) et comme $a\not \equiv 0$ ils sont  bijectifs 
de $H^{s+1}(M)$ sur $H^s(M)$ pour tout $s\in \mathbb R$.
Comme $Re(\mathcal D)=-aId$, on a 
$\sup_{t\geq 0}\Vert e^{t\mathcal D}\Vert_{L^2}\leq 1$.
De plus, comme $\mathcal D$ commute \`a $e^{t\mathcal D}$
et que pour $k$ entier la norme $\Vert \mathcal D^k v\Vert_{L^2}$ est \'equivalente \`a la norme $H^k$, pour tout $s\in \mathbb R$, il existe $C_s$
tel que 
\begin{equation}\label{5.14}
\sup_{t\geq 0}\Vert e^{t\mathcal D}\Vert_{H^s(M)}\leq C_s 
\end{equation}
Rappelons que la solution $v(t)$ de \eqref{5.13} avec $v(0)=v$ v\'erifie
\begin{equation}\label{5.15}
\ba
& v(t)=e^{t\mathcal D}v+r(t,v)\\
&\sup_{t\in [0,T]}\Vert r(t,v)\Vert_{L^2}\leq C_T\Vert v \Vert_{H^{-1}}
\ea\end{equation}
La preuve de \eqref{5.15} est standard: on commence par v\'erifier que si 
$b_{-1}$ est un op\'erateur pseudo-diff\'erentiel de degr\'e $-1$ et de symbole
principal $a(x)\vert\xi\vert_x^{-1}/{2i}$, alors en posant  
$I=Id+B_{-1}$, $B_{-1}=\begin{pmatrix}
0 &b_{-1} \\ 
-b_{-1} & 0
\end{pmatrix}$, 
on a $I(\mathcal D + \mathcal R)=\mathcal D I+ \widetilde{\mathcal Q}_{-1}$ o\`u  $\widetilde{\mathcal Q}_{-1}$ est une matrice $2\times 2$ de pseudos de degr\'e $-1$. Comme on peut choisir $b_{-1}$ tel que $I$ soit inversible sur $H^s(M)\times H^s(M)$ pour tout $s$, la fonction 
 $w(t)=Iv(t)$ v\'erifie $\partial_tw(t)=
(\mathcal D + \mathcal Q_{-1})w(t)$ avec $\mathcal Q_{-1}=\widetilde{\mathcal Q}_{-1}I^{-1}$, et la formule de Duhamel prouve qu'on 
$\sup_{t\in [0,T]}\Vert w(t)-e^{t\mathcal D}Iv\Vert_{L^2}\leq C_T\Vert Iv \Vert_{H^{-1}}$. Or \eqref{5.14} (pour $s=0$ et $s=-1$) implique
$\sup_{t\geq 0}\Vert e^{t\mathcal D}Iv-Ie^{t\mathcal D}v\Vert_{L^2}
\leq C\Vert v \Vert_{H^{-1}}$, ce qui prouve \eqref{5.15}.

On notera $P^{(2)}_h$ la probabilit\'e uniforme sur la sph\`ere unit\'e
$S^{(2)}_h=\{(v_0,v_1), \ \Vert v_0\Vert_{L^2}^2+\Vert v_1\Vert_{L^2}^2=1\}$
de $E_h\oplus E_h$, et $\mathbb E^{(2)}_h(f)$
 l'esp\'erance d'une v.a
sur $S^{(2)}_h$. Soit $f_t(v)=\mathcal{E}(U(t)j^{-1}v)$. Alors on a par construction
\begin{equation}\label{5.16}
\widetilde P_h\Big(u\in \widetilde S_h, \ \mathcal{E}(U(t)u)<\varepsilon  \Big)
=P^{(2)}_h\Big(v\in  S^{(2)}_h, \ f_t(v)<\varepsilon  \Big)
\end{equation} 
et d'apr\`es \eqref{5.15} 
\begin{equation}\label{5.17}
\sup_{v\in  S^{(2)}_h, t\in [0,T]}
\vert f_t(v)-\Vert e^{t\mathcal D}v\Vert^2_{L^2}\vert \leq  C_T h   
\end{equation}
avec $C_T$ ind\'ependant de $h\in ]0,1]$.
On a 
$$\Vert e^{t\mathcal D}v\Vert^2_{L^2}=(B_+(t)v_0\vert v_0)_{L^2}
+(B_-(t)v_1\vert v_1)_{L^2}$$
avec $B_{\pm}(t)=e^{td_\pm^*}e^{td_\pm}$. D'apr\`es le th\'eor\`eme
d'Egoroff, les op\'erateurs $B_{\pm}(t)$ sont des pseudos de degr\'e $0$
et de symbole principal homog\`ene de degr\'e $0$ 
$$\sigma_0(B_{\pm}(t))(\rho)=e^{-2t\underline a (t,\rho)}$$
Comme dans la preuve du lemme \ref{lem2.2.1}, on obtient donc
\begin{equation}\label{5.18}
\sup_{t\in [0,T]}\vert \mathbb E^{(2)}_h(f_t)-\int_{S^*(M)}
e^{-2t\underline a (t,\rho)}d\mu(\rho)\vert\leq C_T h
\end{equation}
Soit $m_{t,h}$ la m\'ediane de la v.a $f_t$ sur $S^{(2)}_h$. Comme $f_t$ est (uniform\'ement
en $t\geq 0$ et $h\in ]0,1]$) lipschitzienne sur $S^{(2)}_h$, on obtient comme
dans la preuve du th\'eor\`eme \ref{thm2.1}, qu'il existe $C,c>0$
tels que pour tout $r\in ]0,1]$

\begin{equation}\label{5.19}
\ba
&\sup_{t\geq 0}P^{(2)}_h(\vert f_t(v)-m_{t,h})\vert\geq r)\leq Ce^{-ch^{-d}r^2}\\
&\Rightarrow  \sup_{t\geq 0}\vert \mathbb E^{(2)}_h(f_t)-m_{t,h}\vert\leq\int_0^{+\infty}P^{(2)}_h(\vert f_t(v)-m_{t,h})\vert\geq r)dr \leq C h^{d/2} 
\ea
\end{equation}   
D'apr\`es l'hypoth\`ese \eqref{5.9}, on a 
\begin{equation}\label{5.20}
\limsup_{t\rightarrow \infty}\int_{S^*(M)}e^{-2t\underline a (t,\rho)}d\mu(\rho)\leq \int_{S^*(M)}\limsup_{t\rightarrow \infty}e^{-2t\underline a (t,\rho)}d\mu(\rho)=0
\end{equation}
Soit alors $T_\varepsilon$ tel que $\int_{S^*(M)}e^{-2T_\varepsilon\underline a (T_\varepsilon,\rho)}d\mu(\rho)\leq \varepsilon/8$. Soit  $h_\varepsilon$ tel qu'on ait \`a la fois
$C_{T_\varepsilon}h_\varepsilon\leq \varepsilon/8$ et 
$Ch_{\varepsilon}^{d/2}\leq \varepsilon/4$ . Alors \eqref{5.18} et la deuxi\`eme ligne de \eqref{5.19} impliquent  
pour tout 
$h\in ]0,h_\varepsilon]$, $m_{t,h}\leq \vert \mathbb E^{(2)}_h(f_t)\vert+
\vert \mathbb E^{(2)}_h(f_t)-m_{t,h}\vert\leq \varepsilon/2$, et donc la 
premi\`ere ligne de \eqref{5.19} implique (avec $r=\varepsilon/2$) 
et pour tout 
$h\in ]0,h_\varepsilon]$	 
\begin{equation}\label{5.21}
P^{(2)}_h(\vert f_{T_\varepsilon}(v)\vert\geq \varepsilon)\leq Ce^{-ch^{-d}\varepsilon^2/4}
\end{equation}
La preuve du lemme \ref{lem5.1} est compl\`ete.
\end{proof}

La preuve du th\'eor\`eme \ref{thm5.1} est une cons\'equence facile de
l'estimation \eqref{5.11}: quitte \`a diminuer $h_\varepsilon$, on peut toujours
supposer $Ce^{-ch^{-d}\varepsilon^2}\leq \alpha$ pour $h\in ]0,h_\varepsilon]$. D'apr\`es \eqref{5.0}, on a
\begin{equation}\label{5.22bis}
\lim_{t\rightarrow \infty}\inf_{h\in [h_\varepsilon,1]}\widetilde P_h\Big(u\in \widetilde S_h, \ \mathcal{E}(U(t)u)<\varepsilon  \Big)=1
\end{equation}
Soit alors $T\geq T_\varepsilon$ tel que pour tout $h\in [h_\varepsilon,1]$ on ait
$\widetilde P_h\Big(u\in \widetilde S_h, \ \mathcal{E}(U(T)u)<\varepsilon  \Big)
\geq 1-\alpha$. Comme l'\'energie est d\'ecroissante, on a aussi d'apr\`es
\eqref{5.11} pour tout $h\in ]0,h_\varepsilon]$ 
$\widetilde P_h\Big(u\in \widetilde S_h, \ \mathcal{E}(U(T)u)<\varepsilon  \Big)
\geq 1-\alpha$, ce qui prouve le th\'eor\`eme~\ref{thm5.1}. 
\end{proof}
Pour d\'emontrer le th\'eor\`eme~\ref{thm5.2}, nous allons revisiter la preuve pr\'ec\'edente. On remarque d'abord que par un argument simple de semi-groupe, la norme $\|w(t)\|_{H^{-1} } $ est major\'ee par $C e^{CT} \|v\|_{H^{-1}}$ et donc la  constante $C_T$ appara\^issant dans~\eqref{5.15},et donc aussi dans~\eqref{5.17}, peut \^etre estim\'ee par $Ce^{CT}$. On remarque ensuite que le Th\'eor\`eme d'Egoroff reste vrai pour des temps grands, mais plus petits que $\epsilon \log (1/h)$ (c'est en effet cons\'equence de~\cite[th\'eor\`eme 2.3.6]{Iv}, et des estimations exponentielles triviales sur le comportement en grand temps du flot bicaract\'eristique). Le prix \`a payer pour cela est que les op\'erateurs $B_{\pm} (t)$ sont des op\'erateurs pseudodiff\'erentiels dont les symboles sont dans des classes (l\'eg\`erement) exotiques : $S_{h, h^{- \alpha}, h^{- \alpha}, N}$ selon la terminologie de~\cite[Chapter 1]{Iv}, o\`u $\alpha>0$ peut \^etre choisi arbitrairement petit si $\epsilon >0$ est choisi petit. On en d\'eduit qu'il existe $\delta>0$ tel que 
$$ \sup_{t \in [0, \epsilon \log (1/h)]} | \mathbb{E} _h ^{(2)} ( f_t) - \int_{S^* (M)} e^{-2t \underline{a} (t, \rho)} d\mu( \rho) | \leq C e^{C \epsilon \log( 1/h)} h^{1-2\alpha}\leq C h^{\delta}
$$ o\`u dans la derni\`ere in\'egalit\'e, on a choisi $\epsilon >0$ assez petit.
On d\'eduit maintenant de l'hypoth\`ese qu'il existe $\sigma>0$ tel que
$$ \int_{S^* (M)} e^{-2t \underline{a} (t, \rho)} d\mu( \rho)\leq e^{-2t(\overline{a}- \beta )}+ e^{- S( \beta)t}\leq 2e^{ -2\sigma t}. $$
On d\'efinit la variable al\'eatoire  $g$ sur $S_h^{(2)}$ par la relation 
$$ g = \int_0^{\epsilon \log(1/h)} f_t e^{\sigma t} dt $$
D'apr\`es ce qui pr\'ec\`ede, 
la m\'ediane, $\mathcal{M}$,  de la fonction $g$ est inf\'erieure ou \'egale \`a 
$$C\int_0^{\epsilon \log(1/h)}e^{\sigma t} (e^{- 2\sigma t}+  h^{\delta}) dt,  $$ donc, quitte a diminuer encore $\epsilon >0$, cette m\'ediane est born\'ee. De plus, la norme Lipschitz de la fonction $g$ est clairement born\'ee pour $\epsilon$ petit par 
$$ C\int_{0}^{\epsilon \log(1/h)} e^{\sigma t}dt \leq C h^{-1/2}$$ on en d\'eduit comme pour la preuve de~\eqref{5.15} que 
$$ P_h ( g > \mathcal{M} + r) \leq C e^{-ch^{-d+1} r^2}.$$
Finalement, on remarque que comme l'\'energie de la solution des ondes amorties est une fonction d\'ecroissante, on a l'implication
$$ g \leq A \Rightarrow \forall t \in [0, \epsilon \log (1/h)], 
\quad \int_{0}^t f_{t}e^{\sigma s} ds \leq A \quad \Rightarrow  \quad f_t \leq {\sigma A\over  e^{\sigma t}-1}$$ce qui  termine la d\'emonstration du th\'eor\`eme~\ref{thm5.2}.

\bigskip
Nous allons \`a pr\'esent utiliser les r\'esultats pr\'ec\'edents pour obtenir
un taux de d\'ecroissance presque sur pour des donn\'ees dans l'espace d'\'energie.
On peut identifier les donn\'ees initiales 
 $(u_{0},u_{1})\in {H}^1(M)/\mathbb{C} \oplus L^2(M)$ \`a valeurs r\'eelles
avec $u=\lambda u_{0}+iu_{1}\in L^2(M; \mathbb C) $. On note $\mathcal M$
l'espace des mesures sur $L^2(M; \mathbb C)$ 
 introduites dans l'appendice C, associ\'ees \`a la d\'ecomposition 
 $L^2(M)=\sum_{k} E_{k}$  o\`u
les $E_{k}$ sont les blocs dyadiques standarts et avec $(\alpha_{k})\in l^2$. 
Toute mesure $\mu\in \mathcal M$ d\'efinit ainsi une mesure de probabilit\'e sur 
${H}^1(M)/\mathbb{C} \oplus L^2(M)$.
On a alors 
\begin{thm}\label{th.7} On suppose que la fonction de Birkhoff v\'erifie~\eqref{5.9}. On suppose aussi que les mesures $p_k$ v\'erifient $H_\gamma$, $\gamma>0$ (voir l'appendice C). Par exemple, on peut prendre 
\begin{equation}\label{eq.gauss}
 dp_k = \sqrt{ \frac 2 \pi}  e^{-\frac {r^2} 2} dr
 \end{equation}
 ou 
 \begin{equation}\label{eq.gaussbis}
 dp_k = \delta _{r=1}
 \end{equation}

 Il existe alors un taux de d\'ecroissance $f(t)>0$ qui tend vers $0$ \`a l'infini,  qui ne d\'epend pas du choix de la mesure $\mu\in \mathcal M$,  tel que$$ \mu(\{ U\in \mathcal{H}^1; \exists T,  \mathcal{E}(U(t) u) \leq f(t) \ \forall t\geq T\}) =1
$$

\end{thm} 
\begin{remk}
On remarque que le th\'eor\`eme~\ref{th.7} et les choix $dp_k = \delta_{r=1}$, $(\alpha_k)= \delta_{k=k_0}$ impliquent le th\'eor\`eme~\ref{thm5.1}. Nous allons en fait d\'eduire le th\'eor\`eme~\ref{th.7} du  th\'eor\`eme~\ref{thm5.1}. Si on ommet l'uniformit\'e par rapport \`a la famille de mesures $\mathcal{N}$, le r\'esultat est vrai sans hypoth\`ese sur la fonction de Birkhoff. Il est alors juste cons\'equence de la convergence vers $0$ de l'\'energie pour toute donn\'ee initiale. 
\end{remk}
\begin{proof}
 D'apr\`es le th\'eor\`eme~\ref{thm5.1}, pour tout $j\in \mathbb{N}$, il existe $T_j$ tel \be\label{unif} 
 \forall k,  \widetilde{P}_k \bigl(u \in \widetilde{S} _k ; \mathcal{E}(U(T_j)) u > 2^{-j}\bigr) \leq 2^{-j}
 \ee
On peut supposer $T_{j}<T_{j+1}$ et $\lim T_{j}=\infty$, et on pose 
\be\label{unif2}
f(t)=1 \ \text{pour} \ t <T_{0}, \quad  f(t)= \sum_{j } 2^{-j/2}1_{t \in [T_j, T_{j+1}[}
\quad \text{pour} \ t \geq T_{0}
 \ee
On d\'eduit de (\ref{unif}) en utilisant
$ \mathcal{E}(U(T_j) u) \leq 1$ pour $u\in \widetilde{S} _k $
$$ \int_{\widetilde{S} _k} \mathcal{E}(U(T_j) u) dP_{k} \leq 2^{-j}+2^{-j}$$
En \'ecrivant pour $u\in E_{k}$, $u=r\omega$, avec $\omega\in \widetilde{S} _k$ on obtient
(avec 
$C=2\sup_{k}\int r^2dp_{k}<\infty$ d'apr\`es l'hypoth\`ese $(H_{\gamma})$, voir 
l'appendice C)

\begin{equation}\label{eq.decr-unif}
\mathbb{E} ( \mathcal{E}(U(T_j) u) = \int_{0}^\infty r^2(\int_{\widetilde{S} _k} \mathcal{E}(U(T_j) u)dP_{k})dq_{k} \leq 2^{1-j}\int_{0}^\infty r^2dq_{k}\leq C2^{-j}\vert \alpha_{k}\vert ^2
\end{equation} 
On remarque maintenant d'apr\`es le th\'eor\`eme~\ref{thm5.0} qu'il existe $C>0$ tel que pour tout $t\geq 0$ et tout $u$ d'\'energie finie
$$ u = \sum_k u_k, u_k \in E_k \Rightarrow \mathcal{E}(U(t)u) \leq C \sum_k \mathcal{E}(U(t)u_k).$$
On en d\'eduit d'apr\`es~\eqref{eq.decr-unif}
\be
\mathbb{E} ( \mathcal{E}(U(T_j) u)\leq C\sum_k \mathbb{E} ( \mathcal{E}(U(T_j) u_{k})
\leq C 2^{-j}\|(\alpha_k)\|_{\ell^2}^2
\ee
D'apr\`es l'in\'egalit\'e de Bienaym\'e-Tchebitchev, on obtient $$ \mu (\{ U\in \mathcal{H}^1; \mathcal{E}(U(T_j) u) \geq 2^{-j/2} \}) \leq C 2^{-j/2}\|(\alpha_k)\|_{\ell^2}^2
$$ 
 Notons
 $$A_j= \{ u; \mathcal{E}(U(T_j) u) \leq  2^{-j/2}\}, A^j= \cap_{k\geq j } A_k$$
 La suite $A^j$ est croissante et on a $\mu(A^j) \geq (1-C \|(\alpha_k)\|_{\ell^2}^2 \sum_{k\geq j} 2^{-k/2})= 1- C' 2^{-j/2} $. Si $A=\cup A^j$, on a $\mu(A)= 1$, et pour tout $u\in A$,
 il existe $C$ tel que pour tout $t$ on ait $\mathcal{E}(U(t)u)\leq Cf(t)$. En effet,
 pour $u\in A^j$ on a en utilisant la d\'ecroissance de l'\'energie 
  $$\mathcal{E}(U(t)u)\leq f(t), \qquad \forall t\geq T_j$$
Le th\'eor\`eme~\ref{th.7} est d\'emontr\'e.
  \end{proof}

\subsection{Conservation de la notion d'\'echelle}
Nous d\'emontrons dans cette section le th\'eor\`eme \ref{thm5.0}. 
\begin{proof}
Soit $N\geq 1$ tel que $c-c'/N\geq c/2$. Pour $h'\geq Nh$, on a 
$ch'-c'h\geq (c-c'/N)h'\geq h'c/2$ et pour $h'\leq Nh$, on a 
$ch'-c'h\geq h'r/N$. Avec $r'=\min (c/2,r/N)$ on a donc toujours
$ch'-c'h\geq r'h'$ donc $c/h-c'/{h'}\geq r'/h$. On choisit $b',b,d$ tels que
\begin{equation}\label{5.22}
c-r'<b'<b<c, \quad \text{et} \quad c'<d
\end{equation}
Soient $\psi,\phi$ des fonctions r\'eelles $C^\infty$ sur $\mathbb R$,
\`a valeurs dans $[0,1]$, avec 
$\psi=1$ pr\`es de $]-\infty,c-r']$, $\psi=0$ sur $[b',\infty[$ et
$\phi=1$ pr\`es de $[c,c']$ et $\phi$ \`a support dans $[b,d]$. Les op\'erateurs
$\phi(h\lambda)$, $\psi(h\lambda)$ op\`erent sur 
$H=H^1(M)/{\mathbb C 1}\oplus L^2(M)$
par la formule $\phi(h\lambda)(e_j)=\phi(h\omega_j)e_j$, et ils commutent
aux op\'erateurs $\widetilde \Pi_h, \widetilde \Pi_{h'}, i_h, 
i_{h'}$. Comme on a d'apr\`es \eqref{5.22} $\psi(h\lambda)i_{h'}=i_{h'}$, 
$\phi(h\lambda)i_h=i_h$, 
$\widetilde \Pi_{h'}\psi(h\lambda)=\widetilde\Pi_{h'}$,
$\widetilde \Pi_h\phi(h\lambda)=\widetilde\Pi_h$, il suffit de v\'erifier
que pour tout $N$, il existe une constante $C_N$,
telle que 
\begin{equation}\label{5.23}
\sup_{t\geq 0}\Vert \phi(h\lambda) U(t)\psi(h\lambda)\Vert_\mathcal{E}
+\sup_{t\geq 0}\Vert \psi(h\lambda) U(t)\phi(h\lambda)\Vert_\mathcal{E}
\leq C_N h^N 
\end{equation}
Comme l'adjoint  de $U(t)$ est $U^*(t)=e^{tA_a^*}$, avec
$$A_a^*=-\begin{pmatrix}0 
&Id \\ 
\mathbf{\Delta} 
& 2a 
\end{pmatrix}$$
on a $U^*(t)=e^{-tA_{-a}}$, donc $U^*(t)(u_0,u_1)=e^{tA_{a}}(u_0,-u_1)$. Il suffit donc d'estimer le terme
$\Vert \phi(h\lambda) U(t)\psi(h\lambda)\Vert_\mathcal{E}$.

Nous allons construire un op\'erateur $\Theta_h$, born\'e sur $H$ pour tout
$h \in ]0,1]$ fix\'e, qui commute \`a U(t), et qui v\'erifie
\begin{equation}\label{5.24}
\ba
&\phi(h\lambda)=\phi(h\lambda)\Theta_h+R_{1,h}\\
&\Theta_h\psi(h\lambda)=R_{2,h}\\
\ea\end{equation}
et tel que pour tout $N$, il existe une constante $C_N$,
telle que 
\begin{equation}\label{5.25}
\Vert R_{1,h} \Vert_\mathcal{E}+\Vert R_{2,h} \Vert_\mathcal{E}\leq C_Nh^N
\end{equation} 
On aura alors
\begin{equation}\label{5.26}\ba
&\phi(h\lambda) U(t)\psi(h\lambda)=(\phi(h\lambda)\Theta_h+R_{1,h})
U(t)\psi(h\lambda)=\\
&\phi(h\lambda)U(t)\Theta_h\psi(h\lambda)+R_{1,h}U(t)\psi(h\lambda)=\\
&\phi(h\lambda)U(t)R_{2,h}+R_{1,h}U(t)\psi(h\lambda)
\ea\end{equation}
et comme $\Vert \phi(h\lambda)U(t)\Vert_\mathcal{E}\leq 1$ et
$\Vert U(t)\psi(h\lambda)\Vert_\mathcal{E}\leq 1$, on obtiendra d'apr\`es \eqref{5.25}
\begin{equation}\label{5.27}
\sup_{t\geq 0}\Vert \phi(h\lambda) U(t)\psi(h\lambda)\Vert_\mathcal{E}\leq C_N h^N
\end{equation}

On effectue maintenant une r\'eduction semi-classique en posant
$\tau(u_0,u_1)=(u_0,hu_1)$ et en introduisant la norme semi-classique
\begin{equation}\label{5.28}
\Vert (w_0,w_1)\Vert^2_{E_h}={1\over 2}\int_M 
(\vert h\nabla_x w_0\vert^2+\vert w_1\vert^2)d_gx
\end{equation}
de sorte qu'on a $\Vert \tau(u)\Vert^2_{E_h}=h^2\Vert u\Vert^2_\mathcal{E}$, 
donc pour tout op\'erateur $R$, 
$\Vert \tau R \tau^{-1}\Vert_{E_h}=\Vert R\Vert_\mathcal{E}$, et
\begin{equation}\label{5.29bis}
\tau U(t)\tau^{-1}=e^{itB/h}
\end{equation}
avec 
\begin{equation}\label{5.29}
B=-ih\tau A_a\tau^{-1}=\begin{pmatrix}
0 &-i \\ 
-ih^2\mathbf{\Delta} & 2iha 
\end{pmatrix}
\end{equation} On a pour $z\in \mathbb C$
\begin{equation}\label{5.30}
Im((z-B)w\vert w)_{E_h}=Im(z)\Vert w \Vert^2_{E_h}
-2h\int_M a\vert w_1\vert^2/2 d_gx
\end{equation}
Si $U_h$ est la bande $U_h=\{z\in \mathbb C, Im(z)\in [0,2h\Vert a \Vert_{L^\infty}]\}$, la r\'esolvante $(z-B)^{-1}$ existe donc pour $z\in \mathbb C\setminus U_h$,
et v\'erifie
\begin{equation}\label{5.31}
\Vert (z-B)^{-1}\Vert_{E_h}\leq {1\over dist(z,U_h)}
\end{equation}
On pose 
\begin{equation}\label{5.32}
f_{\varepsilon}(z)={1\over \sqrt{2\pi\varepsilon}}
\int_Je^{-(z-x)^2/{2\varepsilon}}dx, \quad J=[-d,-b]\cup [b,d]
\end{equation} 
La fonction  $f_{\varepsilon}(z)$ est holomorphe dans $\mathbb C$ et v\'erifie
pour tout $z\in \mathbb C$,  
$\vert f_{\varepsilon}(z)\vert\leq e^{Im(z)^2/{2\varepsilon}}$.
De plus,on a  pour tout $z\in\mathbb C$
 
\begin{equation}\label{5.33}
\vert f_{\varepsilon}(z)\vert\leq {C\over \sqrt\varepsilon}
e^{Im(z)^2/{2\varepsilon}-dist(Re(z),J)^2/{2\varepsilon}}
\end{equation} 
On d\'efinit alors l'op\'erateur $\Theta_h$ par
\begin{equation}\label{5.34}
\Theta_h={1\over 2i\pi}\int_\gamma f_{\varepsilon}(z)(z-B)^{-1}dz
\end{equation}
Dans \eqref{5.34} on choisit le contour $\gamma$ de la forme $\gamma=\gamma_-\cup \gamma_+$. $\gamma_+$ est la r\'eunion du segment 
$[-r_0,r_0]+ir_{+}$, 
$r_{+}=2h\Vert a \Vert_{L^\infty}+\sqrt{\varepsilon}$, et des deux demi-droites 
$r_0+ir_{+}+\rho e^{i\theta_0}, \rho\geq 0$ et 
$-r_0+ir_{+}-\rho e^{-i\theta_0}, \rho\geq 0$, avec 
$\theta_0>0$ petit et $r_0>d$ grand; on oriente $\gamma_+$
de  droite \`a gauche.
$\gamma_-$   est la r\'eunion du segment 
$[-r_0,r_0]-i\sqrt{\varepsilon}$, 
et des deux demi-droites 
$r_0-i\sqrt{\varepsilon}+\rho e^{-i\theta_0}, \rho\geq 0$ et 
$-r_0-i\sqrt{\varepsilon}-\rho e^{i\theta_0}, \rho\geq 0$; 
on oriente $\gamma_-$
de gauche \`a droite. On choisit le param\`etre $\varepsilon$
sous la forme $\varepsilon=h^\nu$
avec $\nu\in ]0,1/4]$. On a alors d'apr\`es \eqref{5.31} et \ref{5.33},
pour une constante $C$ ind\'ependante de $h$
\begin{equation}\label{5.35} 
\Vert \Theta_h\Vert_{E_h}\leq C \varepsilon^{-1}
\end{equation} 
et $\Theta_h$ commute \`a $e^{itB/h}$. 
On va construire des op\'erateurs $Q_{h,k}(z),k\geq 0$ et $R_{h,N}(z), N\geq 1$, holomorphes en $z$ pr\`es de $\gamma$, et qui v\'erifient pour tout $N\geq 1$
avec $C_N$ ind\'ependant de $h,z$
\begin{equation}\label{5.36}\ba
&(z-B)\sum_{k=0}^{N-1}({h\over i})^kQ_{h,k}(z)=Id-h^NR_{h,N}(z)\\
&\Vert R_{h,N}(z)\Vert_{E_h}\leq C_N\varepsilon^{-(1+2N)}
\ea\end{equation}
Posons $\Theta_{h,k}={1\over 2i\pi}\int_\gamma f_\varepsilon (z)Q_{h,k}(z)dz$ et
$\Theta_{h}^N=\sum_{k=0}^{N-1}({h\over i})\Theta_{h,k}$. On a 
pour tout $N\geq 1$ d'apr\`es \eqref{5.31}, \eqref{5.33},
\eqref{5.34} et \eqref{5.36}, en utilisant $\varepsilon=h^\nu$
\begin{equation}\label{5.37}\ba
&\Theta_h=\Theta_{h}^N+ {h^N\over 2i\pi} \int_\gamma
f_{\varepsilon}(z)(z-B)^{-1}R_{h,N}(z)dz\\
&\Vert {h^N\over 2i\pi} \int_\gamma
f_{\varepsilon}(z)(z-B)^{-1}R_{h,N}(z)dz\Vert_{E_h}\leq C'_N h^{N(1-2\nu)-2\nu}
\ea
\end{equation}
Comme on a $\nu<1/2$, pour obtenir \eqref{5.24} et \eqref{5.25}, 
il suffit de prouver que pour tout $N_2$, il existe $N_1\geq N_{2}$ et $C$
ind\'ependant de $h$ tels que  
\begin{equation}\label{5.24bis}
\ba
&\phi(h\lambda)=\phi(h\lambda)\Theta_h^{N_1}+R_{1,h,N_1}\\
&\Theta_h^{N_1}\psi(h\lambda)=R_{2,h,N_1}\\
&\Vert R_{1,h,N_1} \Vert_{E_h}+\Vert R_{2,h,N_1} \Vert_{E_h}\leq C h^{N_2}
\ea\end{equation}
 
Nous allons construire les op\'erateurs $Q_{h,k}(z)$ et $R_{h,N}(z)$ en utilisant le calcul $h$-pseudo\-diff\'erentiel classique (voir \cite{M}).
Pour $m$ entier relatif, on notera $\mathcal M^m$ l'espace des op\'erateurs matriciels
$2\times2$
$$\begin{pmatrix}
A_1(z,x,hD) & A_2(z,x,hD) \\ 
A_3(z,x,hD) & A_4(z,x,hD) 
\end{pmatrix}$$
o\`u  $A_j(z,x,hD)$ est un polyn\^ome de $z$ de degr\'e $\leq m+p_j$, avec 
$p_1=p_4=0,p_2=-1,p_3=1$, et
$A_j(z,x,hD)=\sum_{l=0}^{l=m+p_j}z^lA_{j,l}(x,hD)$ avec 
$A_{j,l}\in \mathcal E^{m-l+p_j}_{cl}$. Par convention, un polyn\^ome
de degr\'e strictement n\'egatif est nul, de sorte que $\mathcal M^m$ est
r\'eduit \`a $\{0\}$ pour $m\leq -2$, et les \'elements de $\mathcal M^{-1}$
sont de la forme $\begin{pmatrix}
0 & 0 \\ 
A(x,hD) & 0
\end{pmatrix}$ avec $A(x,hD) \in  \mathcal E^{0}_{cl}$.
On a pour tout $m,m'$ $\mathcal M^m\mathcal M^{m'}\subset \mathcal M^{m+m'}$
et pour $m\leq m'$ $\mathcal M^m\subset \mathcal M^{m'}$. Dans une carte 
locale, on notera $\mathcal S^m$ les symboles complets des 
\'el\'ements de $\mathcal M^m$. On a $zId \in \mathcal M^1$
et d'apr\`es \eqref{5.29}, $B\in \mathcal M^1$. On posera
$\beta_1(x,\xi)=\begin{pmatrix}
0 &-i \\ 
i\vert\xi\vert_x^2 & 0 
\end{pmatrix}\in \mathcal S^1$, et $\delta=z^2-\vert\xi\vert_x^2$ de sorte que 
\begin{equation}\label{5.38}
(z-\beta_1)^{-1}={1\over \delta}\begin{pmatrix}
z &-i \\ 
-i\vert\xi\vert_x^2 & z 
\end{pmatrix} \in {1\over \delta}\mathcal S^1
\end{equation}
Soit $V_l$ un recouvrement
fini de $M$ par des ouverts de cartes, $\theta_l, \theta'_l\in C_0^\infty(V_l)$
avec $\theta'_l$ \'egal \`a $1$ au voisinage du support de $\theta_l$ et
$\sum_l \theta_l=1$. Dans chaque ouvert de carte $V_l$ on choisit des coordonn\'ees $(x_1,...,x_d)\in \mathbb R^d$, et on d\'efinit les op\'erateurs
$Q_{h,k}(z)$ par la formule
\begin{equation}\label{5.39}
Q_{h,k}(z)=\sum_l \theta'_l Op(q_{k}(z,x,\xi))\theta_l
\end{equation}
o\`u les $q_{k}(z,x,\xi)$ sont d\'efinies par les formules de
r\'ecurrence qui d\'efinissent l'inverse formel de $(z-B)$, soit
\begin{equation}\label{5.40}\ba
&q_{0}(z,x,\xi)= (z-\beta_1)^{-1}\\
&q_{n}(z,x,\xi)=(z-\beta_1)^{-1}\Big(\sum_{\vert\alpha\vert+k=n,\vert\alpha\vert\geq 1}{1\over \alpha !}\partial_\xi^{\alpha}\beta_1\partial_x^{\alpha}q_k
+ \sum_{\vert\alpha\vert+k=n-1}{1\over \alpha !}\partial_\xi^{\alpha}\beta_0\partial_x^{\alpha}q_k \Big)
\ea\end{equation}
o\`u $\beta_1+h\beta_0/i$ est le symbole complet de $B$ dans la carte locale,
avec $\beta_0\in \mathcal S^0$. D'apr\`es \eqref{5.29} on a pour tout
$\alpha$, $\partial_\xi^{\alpha}\beta_1\in \mathcal S^{1-\vert\alpha\vert}$,
$\partial_\xi^{\alpha}\beta_1=0$ pour $\vert\alpha\vert >2$ 
et $\partial_\xi^{\alpha}\beta_0\in \mathcal S^{-\vert\alpha\vert}$, 
$\partial_\xi^{\alpha}\beta_0=0$ pour $\vert\alpha\vert\geq 1$. De plus, on 
a pour $k\geq 0$ et tout $\alpha$, 
$\partial_x^\alpha (\delta^{-(1+2k)}\mathcal S^{1+3k})\subset 
\delta^{-(1+2k+\vert\alpha\vert)}\mathcal S^{1+3k+2\vert\alpha\vert}$.
On montre alors facilement
par r\'ecurrence qu'on a pour tout $n\geq 0$
\begin{equation}\label{5.41}
q_n= {p_{n} \over \delta^{1+2n}}, \ \ \ p_{n}\in \mathcal S^{1+3n}
\end{equation}
L'op\'erateur $R_{h,N}(z)$ est alors d\'efini par la formule \eqref{5.36} et a donc pour symbole
\begin{equation}\label{5.42}
-r_{h,N}(z,x,\xi)=\sum_{\vert\alpha\vert+k=N}{1\over \alpha !}\partial_\xi^{\alpha}\beta_1\partial_x^{\alpha}q_k
+ \beta_0 q_{N-1} \in \delta^{-(1+2N)}\mathcal S^{2+3N}
\end{equation}
Comme on a $\min_{z\in \gamma,x,\xi}\vert\delta\vert\simeq \varepsilon$ et que
$\vert\delta\vert\simeq \vert z \vert^2+\vert\xi\vert_{x}^2$ pour  $(z,\xi)\rightarrow \infty$,
la deuxi\`eme ligne de \eqref{5.36} r\'esulte de \eqref{5.42}. Comme tous les $q_{n}$ sont
des fractions rationnelles de $z$, on peut d'apr\`es \eqref{5.33} utiliser la formule des
r\'esidus pour calculer les symboles des op\'erateurs $\Theta_{h,k}$. Pour 
$\xi$ pr\`es de $0$ on obtient en utilisant la majoration \eqref{5.33}
 $\Theta_{h,k}(x,\xi)\in \mathcal O (e^{-c/h^{\nu}})$  et pour $\xi\not= 0$ on trouve

\begin{equation}\label{5.43}
\Theta_{h,k}(x,\xi)= {1\over (2k)!}\partial^{2k}_{z=\vert\xi\vert_{x}}
\Big( f_{\varepsilon}(z){p_{k}(z,x,\xi)\over (z+\vert\xi\vert_{x})^{1+2k}}\Big) +
{1\over (2k)!}\partial^{2k}_{z=-\vert\xi\vert_{x}}
\Big( f_{\varepsilon}(z){p_{k}(z,x,\xi)\over (z-\vert\xi\vert_{x})^{1+2k}}\Big)
\end{equation}
Il  r\'esulte alors de \eqref{5.32} et \eqref{5.33} qu'on a 
$\Theta_{h}^N(x,\xi)\in \mathcal O (e^{-c/h^{\nu}})$ pour $\xi$
grand, et pour $\xi$ born\'e 
$$\vert\partial_{x}^\alpha\partial_{\xi}^{\beta}\Theta_{h}^N(x,\xi)\vert
\leq C_{\alpha,\beta}\varepsilon^{-(\vert\alpha\vert+\vert\beta\vert)}
=C_{\alpha,\beta}h^{-\nu(\vert\alpha\vert+\vert\beta\vert)}$$
de sorte que les op\'erateurs $\Theta_{h}^N(x,\xi)$ appartiennent \`a une classe
admissible pour le calcul symbolique. Comme on a $f_{\varepsilon}(z)=
1+\mathcal O(e^{-c/h^{\nu}})$ au voisinage de $[-c',-c]\cup [c,c']$ on d\'eduit du fait que 
$\sum_{k} ({h\over i})^kQ_{h,k}(z,x,\xi)$ est le symbole de la r\'esolvante $(z-B)^{-1}$ qu'on a,
pour $\vert\xi\vert_{x}$ proche de $[c,c']$, $\Theta_{h}(x,\xi)=1+\mathcal O(e^{-c/h^{\nu}})$ et $\Theta_{h,k}(x,\xi)=\mathcal O(e^{-c/h^{\nu}})$
pour $k\geq 1$. On a aussi $f_{\varepsilon}(z)=\mathcal O(e^{-c/h^{\nu}})$ 
au voisinage de $[r'-c,c-r']$. En choisissant
$b$ proche de $c$ et $d$  proche de $c'$, on obtient donc
les formules \eqref{5.24bis} par les r\`egles de calcul symbolique. Ceci termine
la preuve du th\'eor\`eme \ref{thm5.0}.
\end{proof}
  
\section{Application aux  ondes non-lin\'eaires sur-critiques}\label{sec5}
On consid\`ere dans cette section l'\'equation des ondes d\'efocalisante sur une vari\'et\'e compacte $M$ (sans bord) de dimension $3$, pour une nonlin\'earit\'e polynomiale $u^p$, $p$ impair.
\begin{equation}\label{eq.ondes}
(\partial_t^2 - \mathbf{\Delta} ) u + u^{p} =0 , u \mid_{t=0}= u_0, \partial_t u \mid_{t=0} = u_1, \quad (u_0, u_1) \in \mathcal{H}^1(M)=H^1(M)\times L^2 (M)
\end{equation}
On supposera $p\geq 7$, de sorte que l'\'equation (\ref{eq.ondes}) est sur-critique.
On supposera aussi $u$ \`a valeurs r\'eelles. Les m\^emes \'enonc\'es restent vrais avec $u$
\`a valeurs complexes en rempla\c{c}ant $u^p$ par $\vert u\vert^{p-1}u$.
Il est bien connu que l'\'equation pr\'ec\'edente poss\`ede pour toutes donn\'ees initiales 
dans $(H^1 (M)\cap L^{p+1} (M)) \times L^2(M)$ des solutions faibles d\'efinies
pour $t\in [0,\infty[$, qui v\'erifient l'\'equation au sens des distributions, et telles que de plus 
\begin{equation}\label{eq.ineg}
 \mathcal{E}(u) (t) = \int_M \frac{|\nabla u|^2} 2 + \frac{|\partial _t u|^2} 2 + \frac{u^{p+1}} {p+1} dx \leq \mathcal{E}(u)(0).
 \end{equation}
 L'objet de ce paragraphe est de construire "beaucoup" de donn\'ees initiales dans $H^1\times L^2$ pour lesquelles d'une part on sait construire une solution forte de l'\'equation~\eqref{eq.ondes} (sur un petit intervalle de temps $(0,T)$), et d'autre part, on sait que toutes les solutions faibles co\" \i ncident avec cette solution forte sur cet intervalle.  
 Dans cette section, les mesures de probabilit\'es $\mathcal{M}_s$ sont celles construites dans l'appendice C avec le choix $a_{k}=2^k$ de la th\'eorie de Littlewood Paley standard.
 \begin{thm}[Existence]\label{thm.ondes}
Pour tout $p<+ \infty$, et pour toute  mesure de probabilit\'es sur $\mathcal{D}' (M) ^2$, $\mu \in \mathcal{M}_1\times \mathcal{M}_0$, pour $\mu$-presque toute donn\'ee initiale $(u_0, u_1) \in H^1(M) \times L^2(M)$,  il existe $T>0$ et une solution forte de~\eqref{eq.ondes} dans l'espace affine
$$ \bigl(\cos( t\sqrt{- \mathbf{\Delta}}) u_0, \frac{ \sin(t\sqrt{ - \mathbf{\Delta}})}{\sqrt{ - \mathbf{\Delta}}} u_1\bigr)  + C^0((0,T); H^2(M)) \cap C^1((0,T); H^1(M) .$$
Cette solution v\'erifie l'identit\'e d'\'energie
$$ \mathcal{E}(u) (t) = \mathcal{E}(u) (0).$$
et les estim\'ees de Strichartz pour tout $r<+\infty$, 
\begin{equation}\label{eq.strichartz}
 \| u \|_{L^\infty((0,T)\times M)}+ \| u \|_{L^2((0,T); W^{1,r}(M))}+ \| \partial_t u \|_{L^2((0,T); L^r(M))} <+\infty
\end{equation}
De plus on a la borne inf\'erieure suivante sur le temps maximal d'existence $T_{\text{max}}(u_0, u_1)$:
\begin{equation}\label{eq.mintemps}
\exists \delta>0, C>0, c>0 ; \mu(\{(u_0, u_1)\in \mathcal{H}^1(M); T_{\text{max} }(u_0, u_1)< \lambda \}) \leq C e^{-c \lambda^{- \delta}}
\end{equation}
\end{thm}
\begin{proof} On suit la strat\'egie de~\cite{BuTz}, en y incorporant nos nouvelles estimations probabilistes. On cherche la solution sous la forme 
$$ u= \bigl(\cos( t\sqrt{- \mathbf{\Delta}}) u_0, \frac{ \sin(t\sqrt{ - \mathbf{\Delta}})}{\sqrt{ - \mathbf{\Delta}}} u_1\bigr)  + v= u_l +v$$
avec $$ v \in C^0((0,T); H^2(M)) \cap C^1((0,T); H^1(M)) .
$$
On cherche alors $v$ v\'erifiant
\begin{equation}\label{eq.ptfixe}
(\partial_t^2 - \mathbf{\Delta} ) v = -(u_l + v )^p, \qquad v\mid_{t=0} = \partial_t v \mid_{t=0} =0
\end{equation}
soit
\begin{equation}\label{duhamel}
 v(t) = -\int_0^t \frac{ \sin((t-s) \sqrt{ -\mathbf{\Delta}})}{\sqrt{- \mathbf{\Delta}}} (u_l + v )^p(s) ds.
\end{equation}
Soit $r>4$ grand et $T\leq 1$. Notons 
$\Vert u_{l} \Vert_{W^{1,r}}=\Vert u_{l} \Vert_{W^{1,r}((0,1)\times M)}$.
D'apr\`es la proposition~\ref{prop.gdedev}, $\Vert u_{l} \Vert_{W^{1,r}}$ est fini $\mu$-presque s\^urement, et d'apr\`es les injections de Sobolev $H^2(M) \rightarrow L^\infty(M)$ et
 $W^{1,r}((0,1)\times M) \rightarrow L^\infty ((0,1)\times M)$ pour $r> 4$, on a pour $s\in [0,T]$
\begin{equation}\ba
 \| (u_l + v )^p(s)\|_{H^1(M)} &\leq C ( \|u_l(s)\|_{L^\infty} + 
 \|v(s)\|_{L^\infty})^{p-1} (\|u_l(s)\|_{H^1} + \|v(s)\|_{H^1})\\
& \leq C ( \Vert u_{l} \Vert_{W^{1,r}} + \|v(s)\|_{H^2})^{p-1} (\|u_l(s)\|_{H^1} + \|v(s)\|_{H^2})
\ea\end{equation}
et de m\^eme, pour $\sup_{0\leq s \leq T}\|v_{j}(s)\|_{H^2}\leq R$ et $s\in [0,T]$,
\begin{equation}\ba
& \| (u_l + v_{1} )^p(s)-(u_l + v_{2} )^p(s)\|_{H^1(M)} \\
&\leq C(\Vert u_{l} \Vert_{W^{1,r}} + R)^{p-2})(\Vert u_{l} \Vert_{W^{1,r}} + R +\|u_l(s)\|_{H^1})(\|v_{1}(s)-v_{2}(s)\|_{H^1})
\ea\end{equation} 
 Comme on a $\int_{0}^T\|u_l(s)\|_{H^1}ds \leq C\sqrt{T}\Vert u_{l} \Vert_{W^{1,r}}$,
  on d\'eduit du th\'eor\`eme du point fixe qu'il existe $c>0$ tel que l'\'equation~\eqref{eq.ptfixe} admet un unique point fixe dans la boule de rayon $R$ de
$$X_T=C^0((0,T); H^2(M)$$
d\`es que 
$$TR^{p-1}+ \sqrt{T}\Vert u_{l} \Vert_{W^{1,r}}^{p-1}+ \sqrt{T}R^{p-2}\Vert u_{l} \Vert_{W^{1,r}}\leq c, \ \text{et} \ \sqrt{T}\Vert u_{l} \Vert_{W^{1,r}}^{p}\leq cR
$$
Comme il existe $\alpha>0, \beta>0$ petits tels que ces in\'equations ont une solution en $R$
pour $\Vert u_{l} \Vert_{W^{1,r}}\leq \alpha T^{-\beta}$, on  d\'eduit~\eqref{eq.mintemps}
de~\ref{eq.gdedev3} et~\eqref{eq.gdedev4}. A nouveau d'apr\`es~\ref{eq.gdedev3} et~\eqref{eq.gdedev4}, $u_{l}$ v\'erifie~\eqref{eq.strichartz}. De plus, d'apr\`es (\ref{eq.ptfixe}), et puisque $(u_{l}+v)\in L^1((0,T),H^1(M))$, les in\'egalit\'es de Strichartz
pour l'\'equation des ondes en dimension $3$ impliquent pour ${1\over q}+{3\over r}={1\over 2}$
et $q\in ]2,\infty]$, qu'on a $v\in L^q((0,T), W^{1,r}(M))$ et $\partial_{t}v\in L^q((0,T), L^r(M))$, donc $v$ v\'erifie~\eqref{eq.strichartz}. La preuve du th\'eor\`eme \ref{thm.ondes}
est compl\`ete.
\end{proof}
\begin{thm}[Unicit\'e fort-faible]\label{thm.ondes2}
Pour $\mu$-presque toutes donn\'ees initiales $(u_0, u_1)$ dans l'espace d'\'energie, $H^1(M) \times L^2(M)$, on a  $u_0 \in L^{p+1}$ et  toute solution faible de~\eqref{eq.ondes} v\'erifiant l'in\'egalit\'e d'\'energie~\eqref{eq.ineg} co\"\i ncide avec la solution donn\'ee par le th\'eor\`eme~\ref{thm.ondes} sur l'intervalle de temps $(0,T)$ 
\end{thm}
\begin{proof}
 Le th\'eor\`eme~\ref{thm.ondes2} est une cons\'equence directe du r\'esultat de stabilit\'e suivant: 
 \begin{prop}\label{prop.stab}
 Soit $u$ une solution forte de l'\'equation des ondes~\eqref{eq.ondes} d\'efinie sur $[0,T]$, de donn\'ees initiales $(u_0, u_1)$, donn\'ee par le th\'eor\`eme~\ref{thm.ondes}. Il existe $C>0$ tel que pour toute solution (faible) $v$ sur l'intervalle $[0,T]$ de l'\'equation~\eqref{eq.ondes}, de donn\'ees initiales $(v_0, v_1) \in (H^1(M) \cap L^{p+1} (M)) \times L^2(M)$ v\'erifiant
 $$ \forall t\in [0,T] \quad  \mathcal{E}(v) (t) \leq \mathcal{E}(v)(0),$$
 on a 
 $$ \mathcal{E}(u-v) (t) + \|u-v \|^2_{L^2(M)}(t)\leq C \bigl(\mathcal{E}(u-v) (0)+ \|u-v\|^2 _{L^2(M)}(0)\bigr)$$
 \end{prop}
Pour d\'emontrer cette proposition, on  reprend et on affine un r\'esultat similaire de Struwe~\cite{Str} (d\'emontr\'e pour des solutions fortes $C^\infty$), (il faut modifier la preuve pour l'adapter au niveau de r\'egularit\'e des solutions fortes donn\'ees par le th\'eor\`eme~\ref{thm.ondes}). 
On notera $w= v-u$ et $Dw= (\partial_t w, \nabla_x w)$. On a 
$$ (\partial_t ^2 w - \mathbf{\Delta}_x w) + (u+w) ^p - u^p =0$$
et 
$$ \mathcal{E}(v) = \mathcal{E}(u) + I + II$$
avec 
$$ I = \int_M Du \cdot Dw + u^p w$$
$$ II = \int _M \frac {|Dw|^2 } 2 + \frac { (u+w)^{p+1}  - u^{p+1}} {p+1}   -u^p w$$
Un calcul  (justifi\'e par une r\'egularisation de $u$ et un passage \`a la limite) donne
\begin{equation}
\frac d {dt} I(t) = \int_M ( u^p +p u^{p-1} w - (u+w) ^p ) \partial _t u
\end{equation}
On va montrer qu'il existe une fonction $g(t)\in L^2\subset L^1$ telle qu'on ait
 \begin{equation}\label{eq.I}
 \bigl|\frac d {dt} I(t) \Bigr|  \leq C (\mathcal{E}(w)+ \|w\|_{L^2}^2) (t) g(t)
\end{equation}
Supposons provisoirement (\ref{eq.I}) d\'emontr\'e. 
Un calcul simple montre que
$$  \frac { (u+w)^{p+1}  - u^{p+1}} {p+1}   -u^p w \geq \frac 1 C w^{p+1} - C w^2$$ ce qui implique
\begin{equation}\label{eq.II}
 II(t)\geq \frac 1 C \mathcal{E}(w) (t) - C \|w\|^2_{L^2}
(t)
\end{equation} 
Si on revient \`a l'hypoth\`ese de d\'ecroissance de l'\'energie de $v$, on obtient
\be\label{eq.III}
 0 \leq \mathcal{E}(v) (0) - \mathcal{E}(v) (t) \Rightarrow II(t) \leq II(0) + I(0) - I(t)
 \ee
Comme pour $q$ assez grand, $u_0 \in W^{1,q}(M) \subset L^\infty(M)$, on a 
$$\bigl|II(0)\Bigr| \leq \mathcal{E}(w) (0)+C \|w\|_{L^2}^2(0),$$
  ceci implique d'apr\`es~\eqref{eq.I},~\eqref{eq.II}, et  ~\eqref{eq.III},
 
$$
\mathcal{E}(w) (t) \leq  C \Big(\int_0^t (\mathcal{E}(w)+ \|w\|_{L^2}^2) (s)g(s) ds+\|w\|_{L^2}^2 (t)  +  \mathcal{E}(w)(0) +  \|w\|_{L^2}^2 (0)  \Big)
  $$
  mais d'apr\`es l'in\'egalit\'e de Minkowski,
\begin{equation}
\label{eq.mink}
 \|w\|^2_{L^2}(t) \leq C(\int_0^t \|\partial_t w\|_{L^2} ^2(s)ds + \|w\|_{L^2}^2 (0))
 \end{equation}
et donc
$$
(\mathcal{E}(w)+\|w\|_{L^2}^2) (t) \leq  C\Big(\int_0^t (\mathcal{E}(w)+ \|w\|_{L^2}^2) (s)
(1+g(s)) ds+  \mathcal{E}(w)(0)+ \|w\|_{L^2}^2 (0)  \Big)
  $$
Le lemme de Gronwall d\'emontre alors la proposition~\ref{prop.stab}. \\
Montrons \`a pr\'esent (\ref{eq.I}). On note $J(t)=(\mathcal{E}(w)+ \|w\|_{L^2}^2) (t)$.
Le terme $
\frac d {dt} I(t)$ est combinaison lin\'eaire de termes
$\int_M w^l u^{p-l} \partial _t u$ avec $l\in \{2,...,p\}$. D'apr\`es (\ref{eq.strichartz}),
on a $u^{p-l} \partial _t u\in L^2((0,T);L^r(M))$ pour tout $r<\infty$. Il suffit donc de v\'erifier qu'on a pour tout $l\in \{2,...,p\}$, et pour $f\in \cap_{r<\infty}L^2((0,T);L^r(M))$,
\be\label{eq.IV}
\vert \int_M w^lf\vert \leq J(t)g(t), \ \text{avec} \ g\in L^2
\ee
Choisissons pour $j\in \mathbb N$ des $\varphi_{j}$ tels que $\sum_{j}\varphi_{j}^2(2^{-j}s)=1$, $\varphi_{0}$ \`a support dans $[0,3]$, $\varphi_{j}$ \`a support dans $[1,5]$, la famille
des $\varphi_{j}$ \'etant uniform\'ement $C^\infty$. Soit $A_{j}=\varphi_{j}(2^{-j}\sqrt{-\triangle})$. Alors les opd autoadjoints $A_{j}$ sont uniform\'ement en $j$ born\'es sur tous les $L^q$,
$q\in [1,\infty]$, et on a $\sum A_{j}^2=Id$. Il en r\'esulte
$ \int_M w^lf =\sum_{j}\int_M A_{j}(w^l)A_{j}(f)$.
Par hypoth\`ese sur la fonction $f$, on a pour tout $r<\infty$
\be\label{eq.VI}
\Vert A_{j}(f)\Vert_{L^\infty(M)}\leq 2^{jd/r}g_{r}(t),\ \text{avec} \ g_{r}\in L^2
\ee
Soit $\chi(s)\in C^\infty$ \'egal \`a 1 pour $s\leq 1$ et nul pour $s\geq 2$.
On a, avec $\lambda_{j}>0$
\be
A_{j}(w^l)=A_{j}(w^l \chi({w^2\over\lambda_{j}^2}))+A_{j}(w^l(1-\chi)({w^2\over\lambda_{j}^2}))=
I_{j}+II_{j}
\ee
On a avec $C$ ind\'ependant de $j$,
\be\label{eq.VII}
\Vert II_{j}\Vert_{L^1(M)}\leq C \Vert w^l(1-\chi)({w^2\over\lambda_{j}^2})\Vert_{L^1(M)}\leq
{C\over \lambda_{j}^{p+1-l}}\Vert w \Vert^{p+1}_{L^{p+1}(M)}\leq {CJ(t)\over \lambda_{j}^{p+1-l}}
\ee
et aussi, en utilisant $l\geq 2$ et $\Vert w\nabla_{x}w\Vert_{L^1}\leq 
\Vert w\Vert_{L^2}\Vert \nabla_{x}w\Vert_{L^2}$
\be\label{eq.VIII}
\begin{aligned}
&\Vert \nabla_{x}I_{j}\Vert_{L^1(M)}\leq \Vert A_{j}(\nabla_{x}(w^l \chi)\Vert_{L^1(M)}+
\Vert [A_{j},\nabla_{x}](w^l \chi)\Vert_{L^1(M)} \\
&\leq C \lambda_{j}^{l-2}(\Vert w\Vert_{L^2}\Vert \nabla_{x}w\Vert_{L^2}+\Vert w\Vert_{L^2}^2)\leq C \lambda_{j}^{l-2}J(t)
\end{aligned}\ee
En \'ecrivant 
\be\label{eq.X}
\int_M w^lf = \int_{M}A_{0}(w^l)A_{0}(f) +\sum_{j\geq 1}\int_{M} II_{j}(w^l)A_{j}(f) + 
\sum_{j\geq 1}\int_{M} \triangle I_{j}(w^l)\triangle ^{-1}A_{j}(f)
\ee
on obtient donc en utilisant (\ref{eq.VI}), (\ref{eq.VII}) et (\ref{eq.VIII}),
et $\vert \int_{M}A_{0}(w^l)A_{0}(f)\vert \leq CJ(t)g_{r}(t)$
\be\label{eq.Xbis}
\vert \int_M w^lf\vert \leq CJ(t)g_{r}(t)\Big(1 +\sum_{j\geq 1} {2^{jd/r}\over \lambda_{j}^{p+1-l}} + 
\sum_{j\geq 1} 2^{-j}2^{jd/r}\lambda_{j}^{l-2}\Big)
\ee
ce qui prouve (\ref{eq.IV}) avec le choix $\lambda_{j}=2^{ja}$, $a \in ]0, {1\over p-2}[$  et $r$ grand. 
  \end{proof}
\section{Cas des vari\'et\'es \`a bord}
Dans le cas o\`u la vari\'et\'e $M$ \`a un bord tel que $M\cup \partial M $ est compacte, une bonne partie des r\'esultats expos\'es dans cet article reste vraie si on consid\`ere le laplacien sur $(M,g)$ avec conditions de Dirichlet ou Neumann. En effet, la formule de Weyl avec reste pr\'ecis\'e~\eqref{2.1bis} reste vraie d'apr\`es les travaux d'Ivrii~\cite{Iv}, donc les minorations/majorations de~\eqref{eq.mino2} aussi. En ce qui concerne l'estim\'ee
$e_{x,h}\leq C_{0}N_{h}$, elle est cons\'equence dans le cas avec bord,  des travaux de Sogge~\cite{So02} pour les conditions aux limites de Dirichlet, puis Smith-Sogge~\cite{SmSo} pour les conditions de Neuman.  En effet, d'apr\`es~\cite[Proposition 2.3]{So02}, on obtient
\begin{prop}~\label{prop.spectral} Supposons que $a_h=1 , b_h = 1+ h$. Alors il existe $C>0$ tel que pour tout $x\in M$ et tout $0<h\leq 1$, 
$$ \|\Pi_h\|_{L^1(M) \rightarrow L^\infty(M)} \leq C h^{1-d}.$$
\end{prop}
Il suffit ensuite dans le cas g\'en\'eral de recouvrir l'intervalle $(a_h, b_h)$ par un nombre fini (d'ordre $(b_h -a_h )/h$) d'intervalles de taille $h$, d'appliquer la proposition~\ref{prop.spectral} \`a chacun de ces intervalles pour obtenir dans le cas g\'en\'eral
$$ \|\Pi_h\|_{L^1(M) \rightarrow L^\infty(M)} \leq C h^{-d}(b_h - a_h).$$
Puis, le noyau de l'op\'erateur $\Pi_h$ \'etant donn\'e par
$$ K_h (x,y)= \sum_{k\in I_h} e_k(x) \overline{e_k(y)},$$
on obtient
$$\sup_{x\in M}|e_{x,h}|\leq \|K_h\|_{L^\infty(M\times M)} = \|\Pi_h\|_{L^1(M) \rightarrow L^\infty} \leq C h^{-d}(b_h - a_h).$$ 
 On en d\'eduit que les r\'esultats des sections~\ref{sec2.1},~\ref{sec2.3} ainsi que la remarque~\ref{rem.4} restent valides dans ce cadre. Il est possible que les r\'esultats de la section~\ref{sec5} puissent aussi \^etre \'etendus \`a ce cadre, mais la plage des estimations de Strichartz valides dans ce cadre des probl\`emes aux limites \'etant plus r\'estreinte (voir~\cite{BuLePl, Iva, BlSmSo}), cela n\'ecessiterait une analyse plus pr\'ecise.
\appendix

\section{Calcul des probabilit\'es sur les sph\`eres}\label{sec5.1}
On note $L(dx)$ la mesure de Lebesgue
sur $\mathbb R^N$ et $p_N$ la probabilité uniforme sur la sphère unité
$S(N)$ de $\mathbb R^N$ . La probabilité $p_N$  est
la mesure image de $\Pi_{1\leq j\leq N} {1\over \sqrt{2\pi}}e^{-\vert x_{j} \vert^2/2}L(dx_{j})$ par
l'application $\pi$
\begin{equation}\label{3.1}
x=(x_1,...,x_N) \mapsto \pi(x)=a=(a_1,...,a_N) \in S(N), \quad a_j={x_{j} \over \sqrt {\sum x_{l}^2}}
\end{equation}
En effet,  $\pi_\ast\big(\Pi_{1\leq j\leq N} {1\over \sqrt{2\pi}}e^{- x_{j}^2/2}L(dx_{j})\big)$
est une probabilité sur $S(N)$ invariante par les isométries de $S(N)$, puisque
pour tout $g\in SO(N)$, $g\pi=\pi g$, et  
$\Pi_{1\leq j\leq N} {1\over \sqrt{2\pi}}e^{-x_{j}^2/2}L(dx_{j})=
(2\pi)^{-N/2}e^{-|x|^2/2}L(dx)$
est invariante par l'action de $SO(N)$. 

Soit $M\geq 1$ fixé . Pour $N_j\geq 1$,  soit $N=N_1+...+N_M$.
La mesure de Lebesgue sur $\mathbb R^N=\Pi_j\mathbb R^{N_j}$ est
donnée dans les coordonnées 
$x=(x_1,...,x_{M}), \  x_{j}=\rho_{j}\omega_{j}\in \mathbb R^{N_j}$, avec $
 \rho_j>0, \  
\omega_{j}\in S(N_j)$ 
par la formule
\begin{equation}\label{3.2}
\Pi_{j=1}^{M} \rho_j^{N_j-1}d\rho_j \otimes c_{N_j} p_{N_j}(\omega_j)
\end{equation}
o\`u $c_{N}$ est le volume de $S(N)$.
En particulier, l'application $\vert \pi\vert_M$, de
$S(N)$ dans la sphère réelle $S(M-1)\subset \mathbb R^M$ 
$$x\mapsto \vert \pi\vert_M (x)=
(\vert x_1\vert, ..., \vert x_M\vert)$$
envoie $p_N$ sur la probabilité
\begin{equation}\label{3.3}
(\vert \pi\vert_M)_\ast p_N={\Pi_{j}c_{N_j}\over c_N}\Pi_{j} \ 
1_{\rho_j\geq 0} \ \rho_j^{N_j-1}d\sigma(\rho)
\end{equation}
où $d\sigma(\rho)$ est la mesure de Lebesque sur $S(M-1)$.

En choisissant $M=2,N_1=2, N_2=N-2$, on a donc pour $x=(x_1,x_{2})\in S(N)$,
et pour $t=\cos\theta_0\in [0,1], \theta_0\in [0,\pi/2]$,
\begin{equation}\label{3.3ter}
p_N(\vert x_1\vert>t)={c_2c_{N-2}\over c_N}\int_0^{\theta_0}\cos\theta
(\sin\theta)^{N-3}d\theta={c_2c_{N-2}\over c_N(N-2)}(\sin\theta_0)^{N-2}
\end{equation}
En choisissant $\theta_0= \pi/2$, on obtient 
\begin{equation}\label{3.3quar}
p_N(\vert x_1\vert>0)={c_2c_{N-2}\over c_N(N-2)} =1
\end{equation}
d'o\`u 
\begin{equation}\label{3.3bis}
p_N(\vert x_1\vert>t)={\bf 1}_{t\in [0,1[}(1-t^2)^{\frac N 2-1}
\end{equation}

Rappelons aussi  le r\'esultat suivant de concentration de la mesure (voir par exemple~\cite[Theorem 2.3 et (1.10), (1.12)]{Led})
 \begin{prop}\label{concentration}
 Considerons une fonction $F$ Lipschitz sur la sph\`ere $\mathbb{S}^d=S(d+1)$ (munie de sa distance g\'eod\'esique naturelle et de la mesure de probabilit\'e uniforme, $\mu$). On d\'efinit sa m\'ediane $\mathcal{M}(F) $ par la relation
 $$ \mu( F \geq \mathcal{M} (F)) \geq \frac 1 2 , \qquad \mu( F \leq \mathcal{M}(F)) \geq \frac 1 2.  $$
 Alors, pour tout $r>0$, 
 \begin{equation}\label{eq.conc}
 \mu ( |F- \mathcal{M} (F)| >r ) \leq 2 e^{-(d-1) \frac {r^2} { 2 \|F\|_{\text{Lips}}^2}}
 \end{equation}
 \end{prop}

\section{Calcul $h$-pseudodiff\'erentiel}\label{sec5.2}
Nous rappelons ici les bases du calcul  
$h$-pseudo-diff\'erentiel sur $M$, pour lesquelles nous renvoyons \`a
\cite{M}.
Pour $m\in \mathbb R$, soit $S^m$ l'espace des  fonctions $a(x,\xi,h)$ de classe
$C^\infty$ en  $(x,\xi)\in \mathbb R^{2d}$, 
d\'ependantes du  param\`etre $h\in ]0,1]$ telles que pour tout $\alpha, \beta$, il existe  $C_{\alpha,\beta}$ tel que pour tout 
$(x,\xi)\in \mathbb R^{2d}$ et tout  $h\in ]0,1] $
on a
\begin{equation}\label{5.2.1}
\vert \partial^\alpha_x\partial^\beta_\xi a(x,\xi,h)\vert \leq C_{\alpha,\beta}(1+\vert \xi\vert)^{m-\vert\beta\vert}
\end{equation} 
Pour $a\in S^m$, on note $Op(a)$ l'op\'erateur h-pseudodiff\'erentiel agissant sur l'espace de Schwartz  $\mathcal S (\mathbb R^d)$
\begin{equation}\label{5.2.2}
Op(a)(f)(x)=(2\pi h)^{-d}\int e^{i(x-y)\xi/h}a(x,\xi,h)f(y)dyd\xi
\end{equation} 
Rappelons que pour $a\in S^0$, l'op\'erateur $Op(a)$ est uniform\'ement en $h$ 
born\'e sur $L^2(\mathbb R^d)$, et que pour 
$a\in S^m, b\in S^k$, on a $Op(a)Op(b)=Op(c)$ o\`u  $c=a\sharp b\in S^{m+k}$ est donn\'e par l'int\'egrale   oscillante 
\begin{equation}\label{5.2.3}
c(x,\xi,h)=(2\pi h)^{-d}\int e^{-iz\theta/h}a(x,\xi+\theta,h)b(x+z,\xi,h)dzd\theta
\end{equation}
et admet le d\'eveloppement  asymptotique 
\begin{equation}\label{5.2.4}
c(x,\xi,h)=\sum_{\vert \alpha\vert <N}{h^{\vert\alpha\vert}\over i^{\vert\alpha\vert}\alpha !}\partial^\alpha_\xi a(x,\xi,h)\partial^\alpha_x b(x,\xi,h)
+ h^Nr_N(x,\xi,h), \quad r_N\in S^{m+l-N} 
\end{equation}
Le sous espace $S^m_{cl}$ de $S^m$ est l'ensemble des $a(x,\xi,h)\in S^m$ tels 
qu'il existe une suite $a_n(x,\xi)\in S^{m-n}, n\geq 0$ 
telle que pour tout $N$, on a
\begin{equation}\label{5.2.1bis}
a(x,\xi,h)=\sum_{0\leq n<N}(h/i)^na_n(x,\xi) +h^Nr_N(x,\xi,h), \quad r_n\in S^{m-N}
\end{equation} 
D'apr\`es \eqref{5.2.4}, on a $a\sharp b\in S^{m+k}_{cl}$ pour $a\in S^{m}_{cl}$ 
et $b\in S^{k}_{cl}$.

Soit  $e_j(x)\in C^\infty(M), j\geq 0$ une base orthonormale dans 
$L^2(M,d_gx)$ de fonctions propres de  $-\mathbf{\Delta}$ avec
 $-\mathbf{\Delta} e_j=\omega_j^2 e_j$.  Pour toute distribution $f\in \mathcal D'(M)$,   on pose $c_j(f)=\int fe_j d_gx$ de sorte qu'on a $f(x)=\sum_j c_j(f)e_j(x)$, 
la  s\'erie \'etant convergente dans $\mathcal D'(M)$. Pour $s\in \mathbb R$, soit $H^s(M)=(1-\mathbf{\Delta}_g)^{-s/2}L^2(M,d_gx)$
l'espace de Sobolev usuel sur $M$. Pour $f\in \mathcal D'(M)$ on a $f\in H^s(M)$ ssi 
 $\Vert f\Vert^2_{H^s(M)}=\sum_j (1+\omega_j^2)^{s}\vert c_j(f)\vert^2<\infty$.  On utilise aussi les normes semi-classiques $H^s$ d\'efinies  par
 
 \begin{equation}\label{5.2.scl}
  \Vert f\Vert^2_{h,s}=\sum_j (1+h^2\omega_j^2)^{s}\vert f_j\vert^2
  \end{equation}
Une famille d'op\'erateurs $R_h$, $h\in ]0,1]$, op\'erant sur $\mathcal D'(M)$ est d\^ite r\'egularisante ssi pour tout  $s,t,N$,  $R_h$ envoie $H^s(M)$ dans$H^t(M)$ et il existe  $C_{s,t,N}$ tel que pour tout  $h\in ]0,1]$ on a 
 \begin{equation}\label{5.2.5} 
\Vert R_h(f)\Vert_{H^t(M)} \leq C_{s,t,N}h^N\Vert R_h(f)\Vert_{H^s(M)}
 \end{equation} 
Une famille d'op\'erateurs $A_h$, $h\in ]0,1]$ agissant sur $\mathcal D'(M)$, 
appartient \`a l'espace $\mathcal E^m_{cl}$ des op\'erateurs $h$-pseudodiff\'erentiels d'ordre $m$,
ssi pour tout $x_0\in M$, il existe un ouvert de carte $U$ centr\'e
 en $x_0$ et deux fonctions $\varphi, \psi \in C_0^\infty(U)$ \'egales \`a  $1$ pr\`es de  $x_0$ avec  $\psi$ \'egale \`a $1$ pr\`es du support de 
 $\varphi$ telles que $A_h\varphi=\psi A_h\varphi +R_h$, avec  $R_h$ r\'egularisant et il existe $a\simeq \sum_{n\geq 0}
 (h/i)^na_n(x,\xi) \in S^m_{cl}$, tel que dans l'ouvert de carte
 $U$, on a  $\psi A_h\varphi=Op(a)$. Le symbole principal de $A_h$, $\sigma_0(A_h)(x,\xi)$,  est par d\'efinition le premier terme 
  $a_0(x,\xi)$ du d\'eveloppement asymptotique de $a(x,\xi,h)$. C'est une fonction intrins\`eque sur $T^*M$,
et pour toute fonction lisse $\varphi \in C^\infty(M)$, on a
  \begin{equation}\label{5.2.6} 
e^{-i\varphi(x)/h}A_h(e^{i\varphi(x)/h}) =\sigma_0(A_h)(x,d\varphi(x))+\mathcal O(h) 
 \end{equation} 
Le support essentiel
de $A_h$ est dit contenu dans le compact $K$ de $T^*M$ si on a avec 
les notations pr\'ec\'edentes et dans tout ouvert de carte $U$, 
$a\in S^{-\infty}$, et  $a_n(x,\xi)$ est \`a support dans $K$ pour tout $n$. 

\noindent
$\mathcal E_{cl}=\cup_m\mathcal E^m_{cl}$ est l'alg\`ebre des op\'erateurs 
$h$-pseudodiff\'erentiels classiques sur $M$. Pour  $A_h\in \mathcal E^m_{cl}$ et
$B_h\in \mathcal E^k_{cl}$, on a $A_hB_h\in \mathcal E^{m+k}_{cl}$, $\sigma_0(A_hB_h)=\sigma_0(A_h)\sigma_0(B_h)$ et
le commutateur $[A_h,B_h]=A_hB_h-B_hA_h$ v\'erifie $[A_h,B_h]\in h\mathcal E^{m+k-1}_{cl}$,  $\sigma_0({i\over h}[A_h,B_h])=\{\sigma_0(A_h),\sigma_0(B_h)\}$ o\`u
$\{f,g\}$ est le crochet de Poisson . De plus, 
pour tout $A_h\in \mathcal E^m_{cl}$, on a 
$A^*_h\in \mathcal E^m_{cl}$, $\sigma_0(A^*_h)=\overline {\sigma_0(A_h)}$, et 
pour tout $s\in \mathbb R$, 
il existe $C_s$ ind\'ependant de $h\in ]0,1]$ 
tel que
 \begin{equation}\label{5.2.sclbis}
  \Vert A_hf\Vert_{h,s-m} \leq C_s   \Vert f\Vert_{h,s} \quad \forall f\in H^s(M)
  \end{equation}
De plus, si  $A_{h}\in \mathcal E^{-\infty}_{cl}=\cap_m\mathcal E^m_{cl}$ ,
il existe $C$ ind\'ependant de $h\in ]0,1]$ et $p\in [1,\infty]$ tel que

 \begin{equation}\label{5.2.sclter}
  \Vert A_hf\Vert_{L^p(M)} \leq C   \Vert f\Vert_{L^p(M)} \quad \forall f\in L^p(M)
  \end{equation}

\bigskip
\noindent
Rappelons enfin que pour tout $\phi \in C_0^\infty([0,\infty[)$, 
l'op\'erateur $\phi(-h^2\mathbf{\Delta})$ d\'efini par
  \begin{equation}\label{5.2.7} 
\phi(-h^2\mathbf{\Delta})(f)=\sum_j\phi(h^2\omega_j^2)c_j(f)e_j(x)
 \end{equation} 
 appartient \`a  $\mathcal E^{-\infty}_{cl}$, et que son symbole principal  est
 
   \begin{equation}\label{5.2.7bis} 
\sigma_0(\phi(-h^2\mathbf{\Delta}_g))=\phi(\vert \xi\vert_x^2)
 \end{equation} 
o\`u $\vert \xi\vert_x$ est la longueur riemannienne du covecteur $\xi$ en $x$.  Pour une preuve de ce fait, on renvoie \`a \cite{DimassiSjostrand99}.
\section{Quelques propri\'et\'es des mesures} 

\subsection{Mesures sur l'espace d'\'energie}\label{sec.4.3}
Soit $0<a_{0}<a_{1}<...$ une suite strictement croissante telle que 
$\lim_{k\rightarrow \infty}a_{k}=\infty$, et telle qu'il existe $C>1$ tel que
\be
\forall k\geq 0, \quad a_{k+1}\leq Ca_{k}
\ee
On pose $a_{-1}=-1$ et pour tout $k\geq 0$
\be
E_{k}=\{u \in L^2(M); u(x)=\sum_{j\in I_{k}}z_{j}e_{j}(x)\}, 
\quad I_{k}=\{j; a_{k-1}<\omega_{j}\leq a_{k}\}
\ee 
On a alors $L^2(M)=\oplus_{k\geq 0}E_{k}$ et pour $u=\sum_{k\geq 0}u_{k}, u_{k}\in E_{k}$
et tout $s\geq 0$, quitte \`a augmenter $C$ pour avoir $1+a_{0}^2\leq C^2$
\be
C^{-2s}\sum_{k\geq 0}(1+a_{k}^2)^s \Vert u_{k} \Vert_{L^2}^2\leq \Vert u \Vert_{H^s}^2\leq \sum_{k\geq 0}(1+a_{k}^2)^s \Vert u_{k} \Vert_{L^2}^2
\ee
Ces in\'egalit\'es sont renvers\'ees pour $s\leq 0$.

On se donne une suite $(\alpha_{k})_{k\geq 0}$ de r\'eels positifs. On se donne \'egalement pour tout $k\in \mathbb{N}$ une mesure de probabilit\'e sur $\mathbb{R}^+$, $p_k(r)$. On supposera qu'il existe $\gamma >0$ tel que 
 \begin{equation*}\tag{$H_\gamma$}\label{eq.decroi} \exists C, c>0; \forall k \in \mathbb{N},
 \int_\rho^{+\infty}  dp_k(r) \leq C e^{- c \rho^\gamma}
 \end{equation*}
 Par convention, on supposera que $H_{\infty}$ est v\'erifi\'ee si les mesures $dp_k$ sont support\'ees dans un compact fixe (ind\'ependant de $k$). Par exemple $dp_k^i = \delta_{r=r_0}$ v\'erifie $H_\infty$ et  $dp_k^i = \sqrt{\frac 2 \pi} e^{- r^2 /2} dr$ v\'erifient $H_2$.
On note  $q_k$ l'image de $p_{k}$ par l'application $r\mapsto \alpha_k r$.
On munit $E_{k}$ de la norme $L^2(M)$, on note $S_k$ sa sph\`ere unit\'e, et ${P}_k$ la probabilit\'e uniforme sur~$S_k$. On note $\nu_{k}$ la probabilit\'e sur $E_{k}$ image de $q_k \otimes {P}_k$
par l'application $(r,\omega)\mapsto r\omega$ de $\mathbb{R}^+ \times S_k$ dans $E_k$. 
 
On notera $P$ la mesure de probabilit\'e d\'efinie sur l'espace $\Pi_{k\geq 0}E_{k}$
par 
 $$P = \otimes_{k} \nu_k, $$
 et $$\|(\alpha_k)\|^2_s= \sum_{k} \alpha_k ^2 ( 1+ a_{k}^{2})^s \in [0, +\infty].
 $$
On identifiera la suite $U= (u_k)\in \Pi_{k\geq 0}E_{k}$ avec la somme de la s\'erie $u= \sum_k u_k$ (on sera par la suite toujours dans un cadre o\`u cette s\'erie converge dans $\mathcal{D}'(M)$). On notera $\mathcal{M}_s$ l'ensemble des mesures de probabilit\'e ainsi d\'efinies  quand $(\alpha_k)$ d\'ecrit l'ensemble des suites v\'erifiant $\|(\alpha_k)\|_s<+\infty$. 
 On a alors 
 \begin{prop}
 Supposons que 
 $$ \|(\alpha_k)\|_s<+\infty.$$
 Alors la mesure de probabilit\'e $P$ est support\'ee par $H^s(M)$.  R\'eciproquement, supposons que 
  \begin{equation}\label{eq.infini2} 
  \|(\alpha_k)\|_s =+\infty
  \end{equation}
  et que les mesures $p_k$ ne se concentrent pas en $0$:
  \begin{equation}\label{eq.infini}
  \exists \rho>0, \delta<1; \forall k\in \mathbb{N}, p_k ([0,\rho)) \leq \delta
  \end{equation}
  (on remarquera que cette derni\`ere condition est toujours v\'erifi\'ee si les mesures $p_k$ sont identiquement distribu\'ees et non \'egales \`a $\delta_{r=0}$).
  Alors 
  $$ P( \{U= (u_k); \sum_k \|u_k\|_{H^s(M)}^2 <+\infty\}) =0.
  $$
 \end{prop}
\begin{proof}
 On calcule d'abord
\begin{multline} \mathbb{E} (\|U\|_{H^s(M)}^2) = \mathbb{E}(\sum_k \|u_k\|_{H^s(M)}^2 )= \sum_k \mathbb{E}_k (\|u_k\|_{H^s(M)}^2 )\\
\leq C_{s}\sum_k  (1+ a_k ^{2}) ^s \int_{r=0}^{+\infty} r^2 dq_k  \leq C_{s} \sum_k  
(1+ a_k ^{2}) ^s \alpha_k^2 <+\infty
\end{multline}
ce qui d\'emontre que $\|u\|_{H^s}$ est finie presque surement. 
R\'eciproquement, sous les hypoth\`eses~\eqref{eq.infini2} et~\eqref{eq.infini}, on a 
\begin{multline}\label{eq.prodinf}
\mathbb{E} ( e^{-t \|U\|_{H^s}^2}) = \prod_{k=1}^{+\infty} \mathbb{E}_k ( e^{-t \|u_k\|_{H^s}^2})
\leq \prod_{k=1}^{+\infty} \int_0^{+\infty} e^{-c_{s}t r^2 (1+a_k^{2})^s} d q_k\\
= \prod_{k=1}^{+\infty} \int_0^{+\infty} e^{-c_{s}t r^2 \alpha_k^2 (1+a_k^{2})^s} d p_k
\leq \prod_{k=1}^{+\infty} \Bigl(p_k ([0, \rho) + e^{-c_{s}t \rho^2  \alpha_k^2 (1+a_k^{2})^s} (1- p_k([0, \rho)) \Bigr)\\
\leq \prod_{k=1}^{+\infty} \Bigl[ 1- \bigl(1-p_k ([0, \rho)\bigl) \bigl(1- e^{-c_{s}t\rho^2  \alpha_k^2 (1+a_k^{2})^s}\bigl)\Bigr].
\end{multline}
 On remarque maintenant qu'on peut supposer 
 $$\alpha_k^2 (1+a_k^{2})^s\rightarrow_{k\rightarrow+\infty} 0$$
 (car sinon le produit infini est clairement nul et le r\'esultat est \'evident), et  comme 
 $$ 1- e^{-c_{s}t\rho^2  \alpha_k^2 (1+a_k^{2})^s}\sim c_{s}t\rho^2  \alpha_k^2 (1+a_k^{2})^s
 $$
 on a 
 $$ \Bigl[ 1- \bigl(1-p_k ([0, \rho)\bigl) \bigl(1- e^{-c_{s}t\rho^2  \alpha_k^2 (1+a_k^{2})^s}\bigl)\Bigr]\sim \Bigl[ 1-\bigl(1-p_k ([0, \rho)\bigl)c_{s}t\rho^2  \alpha_k^2 (1+a_k^{2})^s\bigl)\Bigr],
 $$ et donc en prenant le logarithme dans~\eqref{eq.prodinf}, on obtient, d'apr\`es~\eqref{eq.infini}, que
 le produit infini est divergent vers $0$. On obtient donc
 $$\mathbb{E} ( e^{-t \|U\|_{H^s}^2})=0,
 $$
 et donc, $P$-presque surement, 
 $\|U\|_{H^s}^2 = + \infty$.
\end{proof}

\subsection{Crit\`ere d'orthogonalit\'e}
Le crit\`ere suivant est du \`a Kakutani~\cite{Ka}
\begin{thm*}[Kakutani]
On consid\`ere deux mesures $\mu_1, \mu_2$ associ\'ees au m\^eme choix de la suite $(a_{k})$, mais \`a des suites $(\alpha_{k,1}), (\alpha_{k,2})$ et $(dp_{k,1}), (dp_{k,2})$ \`a priori diff\'erentes. On rappelle que les mesures $dq_{k,j}$ sont les images des mesures $dp_{k,j}$  par l'application $r\mapsto \alpha_{k,j} r$. Alors les mesures $\mu_1$ et $\mu_2$ correspondantes sont absoluement continues l'une par rapport \`a l'autre si et seulement si le produit infini 
\begin{equation}
\label{eq.prod}
\prod_{k=1}^{\infty} \int_0^{+\infty} \sqrt{ dq_{k,1} dq_{k,2}}
\end{equation}
 est convergent, c'est \`a dire (on remarquera que d'apr\`es l'in\'egalit\'e de Cauchy-Schwarz, chacun des termes dans le produit infini est inf\'erieur \`a $1$) qu'il est non nul.
De plus, si ce produit infini est divergent, alors les mesures $\mu_1$ et $\mu_2$ sont \'etrang\`eres : il existe un ensemble $A$ de $\mu_1$-mesure $1$ et de $\mu_2$-mesure $0$. 
\end{thm*}
Ce crit\`ere garantit que pour \og la plupart \fg des choix $\alpha_k, dp_k$, les mesures obtenues sont mutuellement \'etrang\`eres: par exemple, si on choisit 
 $$dp_{k,1}=dp_{k,2} = 1_{r>0}\sqrt{ \frac 2 \pi} e^{- \frac {r^2} 2} dr,
 $$
 un calcul simple~\cite[Appendix B.1]{BuTz3} montre que le produit~\eqref{eq.prod} est non nul si et seulement si 
 $$\alpha_{k,1} = 0 \Leftrightarrow \alpha_{k,2}=0  \text{ et }\sum_k \Bigl( \frac{ \alpha_{k,1}} {\alpha_{k,2}}-1 \Bigr) ^2 <+\infty
 $$
 \subsection{Densit\'e du support}
\begin{prop}On suppose que les mesures $dp_k$ chargent tous les ouverts de l'intervalle $]0,+\infty[$, que $\|(\alpha_k)\|_s<+\infty$, et que tous les coefficients $\alpha_k$ sont non nuls. Alors le support de la mesure $\mu$ associ\'ee est $H^s(M)$:
$$\forall v \in H^s(M), \forall \epsilon>0, \mu( \{ u\in H^s(M); \|u-v\|<\epsilon \}) >0
$$
\end{prop}
On renvoit \`a~\cite[Appendix B.2]{BuTz3} pour une preuve de ce r\'esultat dans un cadre tr\`es l\'eg\`erement diff\'erent. 
\subsection{Estim\'ees de grandes d\'eviations}\label{sec.6.3}
Nous allons d\'emontrer des estim\'ees de grandes d\'eviations pour les mesures sur l'espace $H^s(M)$. On notera $\langle t \rangle= (1+ t^2)^{1/2}$.  
\begin{prop}\label{prop.gdedev} 
 Soit $\mathbb{P}\in \mathcal{M}_s$. On suppose que la suite $p_k$ utilis\'ee pour construire $\mathbb{P}$ v\'erifie l'hypoth\`ese~\eqref{eq.decroi} de la page~\pageref{eq.decroi}, pour $\gamma>0$. On suppose aussi qu'il existe $D>0$ tel que pour tout $k$ assez grand on ait
  $a_{k+1}-a_{k}\geq D$. Alors il existe $D_{0}$ tel que pour $D\geq D_{0}$,
  tout $2\leq p<+\infty$ et tout $\delta > \frac 1 p$, il existe $c>0$ tel que 
 \begin{equation}\label{eq.gdedev3}
 \mathbb{P}( \{ u= \sum_k u_k ; \| u\|_{H^s}>\lambda \}) \leq  2e^{-c\bigl(\frac{\lambda}{\|\alpha_k\|_s} \bigr) ^{\frac{\gamma}{\gamma+ 1}}}
 \end{equation}
 \begin{equation}\label{eq.gdedev4}
 \mathbb{P}( \{ u= \sum_k u_k ; \|\langle t\rangle ^{-\delta} \cos( t \sqrt{ - \mathbf{\Delta}}) u\|_{W^{s,p}(\mathbb{R} \times M)}>\lambda \}) \leq 2 e^{-c\bigl(\frac{\lambda}{\|\alpha_k\|_s} \bigr) ^{\frac{\gamma}{\gamma+ 1}}}
 \end{equation}
\begin{equation}\label{eq.gdedev5}
 \mathbb{P}( \{ u= \sum_k u_k ; \|\langle t\rangle ^{-(\delta+1)} \frac{\sin( t \sqrt{ - \mathbf{\Delta}})}{\sqrt{ - \mathbf{\Delta}}} u\|_{W^{s+1,p}(\mathbb{R} \times M)}>\lambda \}) \leq 2 e^{-c\bigl(\frac{\lambda}{\|\alpha_k\|_s} \bigr) ^{\frac{\gamma}{\gamma+ 1}}}
 \end{equation} 
 \end{prop}
 \begin{proof}
 Des r\'esultats similaires apparaissent dans un cadre l\'eg\`erement diff\'erent dans~\cite{BuTz}. Les d\'emon\-strations de ces trois estim\'ees sont essentiellement identiques (la premi\`ere n'ayant pas de d\'ependance en temps, tandis que les deux derni\`eres ne diff\'erent que par le comportement de la solution des ondes correspondant \`a la donn\'ee initiale de fr\'equence $0$). Nous nous limiterons donc \`a d\'emontrer~\eqref{eq.gdedev4}. On a $e^{it \sqrt{ - \mathbf{\Delta}}}u= \sum_k e^{it \sqrt{ - \mathbf{\Delta}}}u_k$.
 \begin{lem}\label{lem.decroissance}
 Il existe $C,c>0$ tels que pour tous $(x,t)\in \mathbb{R}_t \times M$
 $$P_{h_k} ( |e^{it \sqrt{ - \mathbf{\Delta}}}u_k|(x,t) >\lambda) \leq C e^{-c\bigl(\frac{\lambda} { \alpha_k} \bigr) ^{\frac{ 2 \gamma}{ \gamma +2}}}
 $$
 \end{lem}
 En effet, quitte \`a remplacer la base orthonormale de $E_k, (e_n)$ par $e^{it { \omega_n}} e_n$, on se ram\`ene au cas $t=0$. Si $D_{0}$ est assez grand, on peut utiliser  le lemme~\ref{lem.est}, avec $h_{k}=a_{k}^{-1}$, donc
 \begin{multline}
  P_{k} ( |u_k|(x) >\lambda)= \int_0^{+\infty}\int_{z\in S_{k}} 1_{r |z\cdot b_{x,h_{k}}|>\lambda} dz dq_k (r)\\
  =\int_0^{+\infty}P_{h_{k}} ( \{ u; |u(x)| \geq \frac \lambda r \})dq_k (r)\leq \int_0^{+\infty}
  e^{- c_{2}\bigl(\frac \lambda {r \alpha_k }\bigr)^2 } dp_k (r)
  \end{multline}
  D'apr\`es l'hypoth\`ese~\ref{eq.decroi}, on conclut si $\gamma= +\infty$ tandis que si $\gamma<+\infty$, on obtient
 $$ P_{k} ( |u_k|(x) >\lambda)\leq e^{-c_{2} \bigl(\frac \lambda {\alpha_k \rho }\bigr)^2 }+ \int_\rho^{+\infty} dp_k (r) \leq e^{- c_{2}\bigl(\frac \lambda {\alpha_k \rho }\bigr)^2 }+ Ce^{-c\rho^\gamma}
 $$ qu'on optimise en choisissant $\rho = ( \lambda / \alpha_k)^{2/ (\gamma+2)}$, ce qui donne le lemme~\ref{lem.decroissance}.
 \begin{lem}\label{lem.independance}
 Soient $(u_k)$ des variables al\'eatoires ind\'ependantes, \`a valeurs r\'eelles et de moments impairs tous nuls. Alors pour tout $q\in \mathbb{N}^*$, 
 $$\mathbb{E} \Bigl( \bigl(\sum_k u_k\bigr)^{2q}\Bigr)\leq q^q \mathbb{E} \Bigl( \bigl( \sum_k (u_k)^{2}\bigr)^q\Bigr)
 $$
 \end{lem}
 \begin{proof}
 On s'inspire de la preuve classique des in\'egalit\'es de Khintchine (voir par exemple~\cite[Th\'eor\`eme 4.6]{LiQu}.
 On calcule
 $$
 \mathbb{E} \Bigl( \bigl(\sum_{k\leq K} u_k\bigr)^{2q}\Bigr) = \sum_{\alpha_1+ \cdots +\alpha_K= 2q} \frac{(2q)!} {\alpha_1! \dots \alpha_K !} \mathbb{E} \Bigl( u_1^{\alpha_1}\dots u_K^{\alpha_k}\Bigr)
 $$
 En utilisant l'ind\'ependance et l'annulation des moments, on remarque que dans la somme ci dessus, les seuls termes non nuls sont ceux pour lesquels tous les $\alpha_i$ sont pairs, soit
$$ \mathbb{E} \Bigl( \bigl(\sum_{k\leq K} u_k\bigr)^{2q}\Bigr) = \sum_{\beta_1+ \cdots +\beta_K= q} \frac{(2q)!} {(2\beta_1)! \dots (2\beta_K) !} \mathbb{E}(u_1^{2\beta_1}\dots u_K^{2\beta_k})
 $$
 Comme $(2\beta)! \geq 2^{\beta} \beta!$ et $(2q)! \leq 2^q q! q^q$, on obtient
 \begin{multline}
 \mathbb{E} \Bigl( \bigl(\sum_{k\leq K} u_k\bigr)^{2q}\Bigr) \leq  \frac{1 }{ 2^q}\sum_{\beta_1+ \cdots +\beta_K= q} \frac{(2q)!} {(\beta_1)! \dots (\beta_K) !} \mathbb{E}(u_1^{2\beta_1}\dots u_K^{2\beta_k})\\
 =\frac{ (2q)!} { 2^q q!} \sum_{\beta_1+ \cdots +\beta_K= q} \frac{q!} {(\beta_1)! \dots (\beta_K) !} \mathbb{E}(u_1^{2\beta_1}\dots u_K^{2\beta_k})\leq q^q \mathbb{E}\Bigl( \bigl(\sum_k (u_k)^2\bigr)^q\Bigr)
 \end{multline} 
 \end{proof}
 On peut maintenant conclure la preuve de la proposition~\ref{prop.gdedev}: Pour $(x,t)$ fix\'es, on a d'apr\`es le lemme~\ref{lem.independance}
 \begin{multline}
 \|\cos(t\sqrt{ - \mathbf{\Delta}}) \sum_k u_k\|_{L^{2q}(dP)}\leq  \sqrt{q} \| \bigl(\sum_k | \cos(t\sqrt{ - \mathbf{\Delta}})u_k|^2\bigr)^{1/2}\|_{L^{2q}(dP)}\\
 \leq C \sqrt q \Bigl(\sum_k \| | \cos(t\sqrt{ - \mathbf{\Delta}})u_k|^2\|_{L^{q}(dP)}\Bigr)^{1/2} = \sqrt q \Bigl(\sum_k \|  \cos(t\sqrt{ - \mathbf{\Delta}})u_k\|^2_{L^{2q}(dP)}\Bigr)^{1/2} \\
 \leq \sqrt q \Bigl( \sum_k \Bigl(\int_0^{+\infty}  2q \lambda^{2q-1} P_k ( |\cos(t\sqrt{ - \mathbf{\Delta}})u_k|(x,t) >\lambda) d \lambda \Bigr)^{\frac 1 q} \Bigr)^{1/2}\\
 \leq  C\sqrt q (q \frac {\gamma+2} {\gamma})^{\frac 1 {2q}} \|(\alpha_k)\|_{\ell^2} \Gamma^{\frac 1 {2q}} \bigl(q \frac {\gamma+2} {\gamma}\bigr) 
 \leq C q^{\frac{\gamma+1} \gamma } \|(\alpha_k) \|_{\ell^2}
 \end{multline}
 o\`u dans la derni\`ere in\'egalit\'e on a utilis\'e le lemme~\ref{lem.decroissance}. 
 Finalement, on obtient que pour tous $2 \leq p \leq 2q <+\infty$, 
 \be\begin{aligned}
& \| \langle t\rangle^{- \delta} \cos(t\sqrt{ - \mathbf{\Delta}}) u\|_{L^{2q}( dP); L^p(\mathbb{R} \times M)} \leq \| \langle t\rangle^{- \delta} \cos(t\sqrt{ - \mathbf{\Delta}})u\|_{ L^p(\mathbb{R} \times M); L^{2q}( dP)}\\
& \leq C \|\alpha_k\|_{l^2} q^{\gamma+1\over\gamma}\| \langle t\rangle^{- \delta}\|_{ L^p(\mathbb{R} \times M)} \leq C q^{\frac{\gamma+1} \gamma } \|\alpha_k\|_{l^2}
 \end{aligned}\ee
 Par l'in\'egalit\'e de Tchebitchev on obtient
 \be
 \mathbb{P} ( \{u; \|\langle t\rangle^{- \delta}  \cos(t\sqrt{ - \mathbf{\Delta}})u \|_{L^p( \mathbb{R}_t \times M} > \lambda \}) \leq ({C q^{\frac{\gamma+1} \gamma } \|\alpha_k\|_{l^2}\over \lambda})^{2q}
 \ee 
et on conclut et le choix $q= {1\over e}({\lambda \over \|\alpha_k\|_{l^2}})^
{{ \gamma \over \gamma+1}}\geq 2 $ si ${\lambda \over C\|\alpha_k\|_{l^2}} $ est assez grand, 
$$
 \mathbb{P} ( \{u; \|\langle t\rangle^{- \delta}  \cos(t\sqrt{ - \mathbf{\Delta}})u \|_{L^p( \mathbb{R}_t \times M} > \lambda \}) \leq e^{-c({\lambda \over \|\alpha_k\|_{l^2}})^
{{ \gamma \over \gamma+1}}} 
 $$
Enfin, si ${\lambda \over \|\alpha_k\|_{l^2}}$ est born\'e, quitte \`a diminuer $c$,
on a $1\leq 2e^{-c({\lambda \over \|\alpha_k\|_{l^2}})^
{{ \gamma \over \gamma+1}}}$,
ce qui termine la preuve de la proposition~\ref{prop.gdedev} dans le cas $s=0$. Le cas g\'en\'eral s'en d\'eduit en remarquant que $\|v\|_{W^{s,p}( \mathbb{R} \times M)} \sim \| \langle |D_t| + \sqrt{ - \mathbf{\Delta}} \rangle ^s v\|_{L^p( \mathbb{R} \times M)}$. 
 \end{proof}
 

\renewcommand{\No}{\kern-.25em\lower.2ex\hbox{\char'27}}
\def\cprime{$'$}

\end{document}